\documentclass[10pt]{amsart}
\usepackage[utf8]{inputenc}
\usepackage{bbm}
\usepackage[T1]{fontenc}
\usepackage[english]{babel}

\newtheorem{theorem}{Theorem}
\newtheorem*{theorem*}{Theorem}

\newtheorem{lemma}{Lemma}[section]
\newtheorem{proposition}[lemma]{Proposition}

\newtheorem{remark}{Remark}[section]

\DeclareMathAlphabet{\mathpzc}{OT1}{pzc}{m}{it}

\usepackage{mathtools}
\usepackage{amsmath}
\usepackage{amsfonts}
\usepackage{blindtext}
\usepackage{hyperref}
\usepackage{url}
\usepackage{amssymb}
\usepackage{bm}
\usepackage{graphicx,xcolor}

\graphicspath{ {./images/} }
\newcommand{\diag}[1]{\begin{pmatrix}
    #1 & \\
    & 1
\end{pmatrix}}

\newcommand{\G}{\mathrm{G}}
\newcommand{\N}{\mathrm{N}}
\newcommand{\A}{\mathbb{A}}
\newcommand{\Q}{\mathbb{Q}}
\newcommand{\Z}{\mathbb{Z}}
\newcommand{\C}{\mathbb{C}}
\newcommand{\p}{\mathfrak{p}}
\newcommand{\q}{\mathfrak{q}}
\newcommand{\oo}{\mathfrak{o}}
\newcommand{\R}{\mathbb{R}}
\newcommand{\GL}{\mathrm{GL}}

\renewcommand{\d}{\mathrm{d}}
\renewcommand{\tilde}{\widetilde}
\renewcommand{\hat}{\widehat}
\newcommand{\bs}{\symbol{92}}
\newcommand{\vph}{\varphi}
\newcommand{\defeq}{\vcentcolon=}

\begin{document}
\title{The second moment of $\mathrm{GL}(2)\times \mathrm{GL}(2)$ Rankin--Selberg $L$--functions in the level aspect}
\author{Jakub Dobrowolski}
\begin{abstract}
    We prove an asymptotic formula with a power-saving error term for a specific weighted second moment of $\mathrm{GL}(2)\times \mathrm{GL}(2)$ Rankin--Selberg $L$--function, $L(1/2,\pi\otimes \pi_0)$ over any number field $F$ where $\pi$ runs over representations with the non-archimedean conductor dividing an ideal which tends to infinity and $\pi_0$ is a fixed cuspidal representation unramified everywhere. The error term shows the square root cancellation under the assumption of the Generalised Ramanujan Conjecture.
\end{abstract}
\maketitle
\section{Introduction}
\subsection{Introduction to the problem} Let $F$ be a number field, and $\q\subset F$ be an integral ideal coprime to the different ideal of $F$.
In this article, we consider the automorphic representations in the conductor-aspect family
\begin{equation}\label{eq: conductor family}
    \mathcal{F}_\q\defeq \{ \pi\,\; \mathrm{automorphic\, representation\, for\,\;} \mathrm{PGL}_2(F) :\,\;  \mathfrak{C}_f(\pi)\mid \q,\, C_{\infty}(\pi)\ll 1\},
\end{equation}
where $\mathfrak{C}_f(\pi)$ and $ C_{\infty}(\pi)$ are the arithmetic conductor and the archimedean conductor, respectively, as in Section \ref{section: Autreps}. We find the asymptotic expansion with a power-saving error term for a specific weighted second moment of $\mathrm{GL}(2)\times \mathrm{GL}(2)$ Rankin--Selberg $L$ function, where one representation is fixed.
More concretely, the result of this paper can be roughly stated as follows:
\begin{theorem*}
     Let $\pi_0$ be a fixed cuspidal representation for $\mathrm{PGL}_2(F)$ that is unramified at all places. One has the following expansion of a weighted second moment over the conductor family \eqref{eq: conductor family}
    \[ \mathbb{E}_{\substack{\pi\in\mathcal{F}_\q\\ \mathrm{generic}}} \Big[ \frac{\left\lvert L(1/2,\pi\otimes \pi_0)\right\rvert^2}{L(1,\pi,\mathrm{Ad})}+\mathrm{Eisenstein} \Big]=c_F L(1,\pi_0,\mathrm{Ad}) \log^3 N(\q)+ P(\q)+O_{\pi_0,\varepsilon} (N(\q)^{-1/2+\theta+\varepsilon}) \]
    where $c_F$ is a constant depending only on $F$ and $P(\q)$ is a degree 2 polynomial in $\log N(\q)$ with coefficients bounded by $(\log\log N(\q))^4$, and $\theta$ is the best bound for the Generalised Ramanujan Conjecture for $\GL(2)$.
\end{theorem*}
For a more precise statement, see Theorem \ref{Theorem}.  

This is the first result that proves an asymptotic expansion with a power-saving error term for any number field with an arbitrary level. The statement of the theorem above is very general and does not assume anything about the way that the ideal $\q$ grows. Finally, we can explicitly separate the main and error term contributions, as mentioned in \eqref{eq: intro reg spec decomp}, where the error term satisfies a square-root cancellation under the Generalised Ramanujan Conjecture for $\GL(2)$.

The problem was first studied in \cite{KowMichVan}, where the authors obtained a similar asymptotic expansion in the case $F=\Q$ and $\q$ being a prime, see Remark \ref{remark: similarity to KMV}. Their treatment of the error term is based on classical methods, whereas our proof relies on the period-theoretic approach.

Later, in \cite{BlomHar}, the authors considered the analogous archimedean problem, looking at the hybrid moment in weight and $t$ aspect. They obtain an asymptotic expansion with a power-saving error term in the case where the fixed representation $\pi_0$ is either cuspidal or Eisenstein \cite[Theorem 1, Theorem 2, Theorem 3]{BlomHar}. 

In \cite{HumphriesKhan2025SecondMoment}, the authors obtained an asymptotic expansion for the second moment of $L(1/2,f\otimes g)$ where $g$ is a holomorphic Hecke cusp form of weight $k_g=k$ and $f$ ranges over an orthonormal basis of cusp forms of the same weight $k$. Since the weight $k_g$ varies with $k$, obtaining a power-saving error term, for now, seems out of reach.

Finally, from a private conversation with Aritra Ghosh, we learned about his ongoing work on the mixed second moment $L(1/2,f\otimes g)L(1/2,g\otimes h)$, where $g$ runs over holomorphic forms of varying level. His goal is to show simultaneous non-vanishing while we focus primarily on obtaining an asymptotic expansion over an arbitrary number field and with an arbitrary level, and hence obtain a better error term in the expansion.

The methods used in this work are based on the theory of integral representations of automorphic $L$--functions. The starting point of the analysis is the simple-looking identity of the form
\[\langle \lvert f_1\rvert^2,\lvert f_2\rvert^2 \rangle = \langle f_1 f_2,f_1 f_2 \rangle,\]
where $f_1,f_2$ are well-chosen automorphic forms specified in Section \ref{section: Set-up}. This equality was first studied in \cite{Venkatesh2010Sparse} and \cite{MVGL2}, where it was used to establish subconvexity for $\GL(2)$ and $\GL(2)\times \GL(2)$ $L$--functions. We closely follow their work, but we need to be more careful when estimating the period side of the equation. To obtain subconvexity, they need to bound the period by $\ll N(\q)^{\varepsilon}$ whereas we care about the asymptotic expansion in powers of $\log N(\q)$ with a power-saving error term.

The main reason why the result of this paper is so general is that the formula \eqref{eq: inner product identity}, when properly expanded, gives one of the reciprocity formulae of the form
\[\int_{\pi} h_1(\pi) L(1/2,\pi\otimes \pi_0\otimes \mathcal{I}(1/2)) \approx \mathrm{deg}+\int_{\pi'} h_2(\pi') \sqrt{L(1/2,\pi'\otimes \pi_0\otimes \tilde{\pi}_0)}\sqrt{L(1/2,\pi'\otimes \mathcal{I}(1/2)\otimes \mathcal{I}(1/2))},\]
where $\mathrm{deg}$ is the degenerate term, $\mathcal{I}(1/2)$ is defined in Section \ref{section: Eis}, and more careful analysis of both sides of the identity above is done in Sections \ref{section: Calculations on the spectral side} and \ref{section: Calculations on the period side}. For a more general discussion on the reciprocity formulae of this form, see \cite{ZachariasI, ZachariasII, Miao}.
\subsection{Overview of the proof}
Fix a number field $F$. For the fixed cuspidal representation $\pi_0$, choose $\vph_0$ a normalised unramified vector at every place. Let $\q=\prod_v \p_v^{r_v}$ be an ideal with norm going to infinity. We define the principal series test vector $f_s$ (see Section \ref{section: Set-up}) such that for $v\mid \q$ the local component restricted to $\mathrm{GL}_2(\oo_v)$ is the identity function on the Hecke subgroup $K_0(\p_v).$ Define the translation matrix $x=(x_v)_v\in \GL_2(\A)$, by letting $x_v=\diag{\varpi_v^{-r_v}}$ for $v\mid \q$ where $\varpi_v$ is a uniformizer, and $x_v=1$ otherwise. The use of this matrix to detect ramification was first considered in \cite{Venkatesh2010Sparse}. We consider the shifted inner product identity
\begin{equation}\label{eq: inner product identity}
    \langle \vph_0 \mathrm{Eis}(f_{1/2})(.\, x), \vph_0 \mathrm{Eis}(f_{1/2})(.\, x) \rangle=\langle \lvert \mathrm{Eis}(f_{1/2})\rvert^2(.\, x),  \lvert\vph_0\rvert^2\rangle
\end{equation}
which was first investigated in this form in \cite{MVGL2}. The left-hand side uses a standard Plancherel formula argument to arrive roughly at an integral
\[\int_{\hat{\mathbb{X}}_{\mathrm{gen}}} \frac{\left\lvert \Lambda(1/2,\pi_0\otimes \tilde{\pi})\right\rvert^2}{\ell(\pi)}\mathbf{1}_{\mathfrak{C}_f(\pi)\mid \q}\, \mathrm{d}\mu_{\text{aut}}(\pi)\]
which is the second moment over the conductor family $\mathcal{F}_\q$ as defined in \eqref{eq: conductor family}. In reality, we are working with the weight $J_\q(\pi)$ (defined in equation \eqref{eq: def of spec weight}), which we justify the validity of in Proposition \ref{Prop: Lower bound on specwght} and Section \ref{section: Spectral Weight}.

For the right-hand side of the equation \eqref{eq: inner product identity}, we consider, instead, the inner product 
\[\langle \overline{\mathrm{Eis}(f_{1/2+s_1})}(.\, x)\mathrm{Eis}(f_{1/2+s_2})(.\, x),  \lvert\vph_0\rvert^2\rangle\]
for some small $s_1,s_2$. We expand it using the regularised spectral decomposition \eqref{eq: reg spec decomp}, which was developed adelically in \cite{MVGL2},
\begin{equation}\label{eq: intro reg spec decomp}
    \langle \mathcal{E}_{s_1,s_2}(.x), \lvert \vph_0\rvert^2\rangle + \int_\pi  \langle \overline{\mathrm{Eis}(f_{1/2+s_1})}(.\, x)\mathrm{Eis}(f_{1/2+s_2})(.x), \vph_{\pi}\rangle_{\mathrm{reg}} \langle \vph_{\pi}, \lvert \vph_0\rvert^2\rangle
\end{equation}
where  $\mathcal{E}_{s_1,s_2}$ is the Eisenstein series formed out of (four) constant terms of $\overline{\mathrm{Eis}(f_{1/2+s_1})}\mathrm{Eis}(f_{1/2+s_2})$, see \eqref{eq: Eis series out of constant terms}. 

Calculation of the first, degenerate term, is done explicitly in Proposition \ref{Prop: Local Zeta integrals}, and after combining local calculations, we get a sum of the form
\[N(\q)^{\overline{s_1}+s_2} F(\overline{s_1},s_2)+N(\q)^{\overline{s_1}-s_2} F(\overline{s_1},-s_2)+N(\q)^{-\overline{s_1}+s_2} F(-\overline{s_1},s_2)+N(\q)^{-\overline{s_1}-s_2} F(-\overline{s_1},-s_2)\]
where $F$ depends mildly on $\q$ and has the triple pole at $s_1=s_2=0.$ Despite this behaviour of $F$, the limit exists and we calculate it as a function of $\q$ in Proposition \ref{Prop: Degenerate term}.

In Proposition \ref{Lem: Regularised triple product calc} we show that the integrand $\langle \overline{\mathrm{Eis}(f_{1/2+s_1})}(.\, x)\mathrm{Eis}(f_{1/2+s_2})(.x), \vph_{\pi}\rangle_{\mathrm{reg}}$ in equation \eqref{eq: intro reg spec decomp}, can be realised as an $L$--function, and in fact gives a power-saving in $N(\q)$. Finally, in Proposition \ref{Prop: Regularised term}, we show that the rest of the integral converges and gives the desired upper bound.
\subsection{Acknowledgments}
We want to thank our PhD supervisor, Subhajit Jana, for the suggestion of the problem, immense support throughout the project, and for carefully reading and commenting on the first draft. We want to thank Ramon Nunes for a careful reading of an earlier draft and useful comments. We would also like to thank Paul Nelson and Valentin Blomer for their feedback on this paper.
\section{Notation and 
Preliminaries}
\subsection{Notations}\label{section: Notations}
We use the standard Vinogradov notation $A\ll B$, which means that there exists a constant $c>0$ such that $\lvert A\rvert <c \lvert B \rvert $. Moreover, we write $A\ll_{\varepsilon} B$ to allow the constant to depend on $\varepsilon.$

We fix a number field $F$ and write $\oo$ for its ring of integers, $\mathfrak{d}$ for the different ideal. For each place of $F$ we write $F_v$ for the completion of $F$ at $v$. If $F_v$ is non-archimedean, we further write $\oo_v$ for the ring of integers with units $\oo_v^{\times}$, $\p_v$ for the maximal ideal generated by a uniformizer $\varpi_v$, and $\mathfrak{d}_v$ for the local different ideal generated by $\lambda_v$. We also fix a norm on $F_v$ so that $\lvert \varpi_v\rvert=N(\p_v)^{-1}$, where we use the standard notation that $N(\p)\defeq [\oo: \, \p]$. Finally normalise the Haar measure on $F_v$ and $F_v^{\times}$ so that $\text{vol}(\oo_v)=1$ and $\text{vol}(\mathfrak{o}_v^{\times})=1$, respectively. 

When $F_v$ is archimedean (either $\R$ or $\C$), we fix a norm on $F_v$ ($\lvert\cdot \rvert_{\R}$ or $\lvert\cdot \rvert_{\R}^2$ respectively) and take the standard Lebesgue measure $\mathrm{d}x$ on $F_v$ and a Haar measure on $F_v^{\times}$ defined by $\mathrm{d}^{\times} x\defeq \frac{\mathrm{d}x}{\lvert x\rvert}$.

Write $\A=\A_F$ for the ring of adeles of $F$ with a fixed additive character $\psi=\psi_{\Q} \circ \text{tr}_{\A_{\Q}\bs \A_F}$ on $F\bs \A$ where $\psi_{\Q}$ is the standard additive character on $\Q\bs \A_{\Q}$.

We write $\xi=\xi_F$ for the completed Dedekind zeta function for the fixed number field $F$ and $\zeta$ for its finite part. The completed zeta function has local factors $\xi_v(s)=\zeta_v(s)=(1-N(\p_v)^{-s})^{-1}$ for a non-archimedean place $v\nmid \mathfrak{d}$, $\xi_v(s)=\zeta_v(s)=N(\mathfrak{d})^{s/2}(1-N(\p_v)^{-s})^{-1}$ for $v\mid \mathfrak{d}$, and $\xi_v(s)=\Gamma_{F_v}(s)$ for an archimedean place $v$, where $\Gamma_\R(s)=\pi^{-s/2} \Gamma(s/2)$ and $\Gamma_\C(s)=2 (2\pi)^{-s}\Gamma(s).$

Similarly, for $\pi$ an automorphic representation of $\mathrm{GL}_2(\A)$  we let $L(s,\pi)$ be an $L$--function attached to $\pi$. We can write it in terms of an Euler product
\[L(s,\pi)=\prod_{v<\infty}(1-\alpha_{1,v} N(\p_v)^{-s})^{-1}(1-\alpha_{2,v} N(\p_v)^{-s})^{-1},\]
where we call $\alpha_{1,v}$, $\alpha_{2,v}$ the local Satake parameters of $\pi$ at $v$. Finally, we write $\Lambda(s,\pi)$ for the completed $L$--function of $\pi$.

For any factorable global function $H$ (e.g., $\xi$ or $\Lambda$) defined by the Euler product $H=\prod_v H_v$ and an ideal $\q$ we write $H_\q\defeq \prod_{v\mid \q} H_v$. Define the prime ideal function
\begin{equation}\label{eq: ideal omega function}
    \omega_F(\q)\defeq  \#\{ \p\mid \q: \; \p \;\;\mathrm{prime\, ideal}\}.
\end{equation}
Finally, we let $\theta$ be the best bound towards the Generalised Ramanujan Conjecture for $\GL(2)$ over the number field $F$, we have $0\leq \theta\leq 7/64$ \cite[Theorem 1.]{BloBru}.
\subsection{Matrix groups}\label{section: Matrix gps}
For any base ring $R$, we write 
\[\G(R)\defeq \mathrm{PGL}_2(R)\]
for the projective linear group, $\N(R):=\big(\begin{smallmatrix}
            1&R\\
            &1
\end{smallmatrix}\big)$ for the maximal unipotent subgroup of $\G(R)$ and $\mathrm{A}(R):=\big(\begin{smallmatrix}
            R^{\times}&\\
            &1
\end{smallmatrix}\big)\subset \G(R)$ for the diagonal subgroup of $\G(R)$. We identify $\N(R)\cong R$ and $\mathrm{A}(R)\cong R^{\times}$ which, in the case when $R=F_v$ or $R=\A$, naturally define Haar measures and character groups. We also define the maximal compact subgroups of $\G(R)$
\[\mathrm{K}(R)=\begin{cases}
    \G(\oo_v);\quad &\text{if $F_v$ is a non-archimedean local field,}\\
    \mathrm{PO}_2;\quad &\text{if $R=\R$,}\\
    \mathrm{PU}_2;\quad &\text{if $R=\C$,}\\
    \prod_v \mathrm{K}(F_v);\quad &\text{if $R=\A$,}
\end{cases}\]
which come with the associated probability Haar measures, denoted by $\mathrm{d}k$.
Finally, we write $\mathbb{X}\defeq \G(F) \bs \G(\A)$ for the automorphic quotient.

We will be integrating on the quotient space $\N(\A)\bs \G(\A).$ For a factorable integrand, which will be the case in our work, we decompose the space as $\N(\A)\bs \G(\A)=\prod_v \N(F_v)\bs \G(F_v)$. For each place, we can either use the Iwasawa decomposition
\[\int_{\N(F_v)\bs \G(F_v)}\, H(g)\,\mathrm{d}g=\int_{F_v^{\times}}\int_{\mathrm{K}(F_v)} H\left(  \begin{pmatrix}
    y&\\
    &1
\end{pmatrix}k\right)\,\mathrm{d}k\,\frac{\mathrm{d}^{\times}y}{\lvert y\rvert}\]
or a Bruhat decomposition
\begin{equation}\label{eq: Bruhat decomp}
    \int_{\N(F_v)\bs \G(F_v)}H(g)\, \mathrm{d}g=\int_{F_v^{\times}}\int_{F_v} H\left(\begin{pmatrix}
    y&\\
    &1
\end{pmatrix}\begin{pmatrix}
    1&\\
    c&1
\end{pmatrix}\right) \mathrm{d}c\,\frac{\mathrm{d}^{\times}y}{\lvert y\rvert}.
\end{equation}

Finally, for a non-archimedean place $v$ and any $r_v >0$ define the Hecke congruence subgroups of $\mathrm{K}(F_v)$
\begin{equation}\label{eq: def of cong subgp}
    \mathrm{K}_{\p_v^{r_v}}=\Big\{k\in \mathrm{K}(F_v):\;\; k \equiv \begin{pmatrix}
            *&*\\
            0&*
        \end{pmatrix}\, (\mathrm{mod}\,\, \p_v^{r_v})  \Big\}.
\end{equation}
which has volume
\[\mathrm{vol}(\mathrm{K}_{\p_v^{r_v}})=N(\p_v)^{-r_v} \frac{\zeta_v(2)}{\zeta_v(1)}\]
with respect to the probability Haar measure on $\mathrm{K}(F_v)$.

\subsection{Automorphic representations}\label{section: Autreps}
We briefly introduce the notation for automorphic representations that we use; for more details, see \cite{JacLang}. Recall the definition of the automorphic quotient, $\mathbb{X}= \G(F) \bs \G(\A)$. By $\widehat{\mathbb{X}}$ we will denote the isomorphism classes of irreducible, unitary representations appearing in the spectral decomposition of $L^2(\mathbb{X})$. It consists of a discrete spectrum of cuspidal representations, and Dirichlet characters $\chi(\mathrm{det})$, and the continuous Eisenstein series spectrum. We attach to $\widehat{\mathbb{X}}$ a natural Plancherel measure $\mathrm{d}\mu_{\mathrm{aut}}$ which is compatible with the invariant probability measure on $\mathbb{X}$. Write $\widehat{\mathbb{X}}_{\mathrm{gen}}$ for the subclass of the generic representations (see \ref{section: Whittaker}), and $\widehat{\mathbb{X}}_{\mathrm{gen}}^{\mathrm{unr}}$ for the subclass of generic representations unramified at all places. Finally, write $C(\pi)$ for the analytic conductor of $\pi$, $\mathfrak{C}_f(\pi)$ for the conductor ideal and $C_{\infty}(\pi)$ for the archimedean conductor, which satisfy the relation $C(\pi)=N(\mathfrak{C}_f(\pi))C_{\infty}(\pi).$

\subsection{Eisenstein series}\label{section: Eis}
We give a very brief overview of the Eisenstein series; for more details, see \cite[Chapter VII]{JacLang}. We start with a Schwartz-Bruhat function $\Phi\in\mathcal{S}(\A^2)$ and construct the associated principal series vector
\begin{equation*}
    f_{\Phi,s}\in \mathcal{I}(s)\defeq \mathrm{Ind}_{\N(\A)\mathrm{A}(\A)}^{\G(\A)}(\lvert \cdot\rvert^s \otimes \lvert\cdot\rvert^{-s})
\end{equation*}
by
\begin{equation}\label{eq: def of princ series vec}
    f_{\Phi,s}(g)\defeq \lvert \mathrm{det}(g)\rvert^s \int_{F^{\times}\bs \A^{\times}} \Phi\left(t (0\; 1)g\right)\lvert t\rvert^{2s}\, \mathrm{d}^{\times}t.
\end{equation}
The integral is defined and absolutely convergent for $\mathrm{Re}(s)>1/2$. We drop the $\Phi$ from the subscript if it is clear from the context.
The vector $f_s\in \mathcal{I}(s)$ satisfies the transformation property
\[ f_{s}\left(\diag{y}g\right)=\lvert y\rvert^s f_s(g).\]
For any $\Phi\in \mathcal{S}(\A^2)$ we define the Fourier transform 
\[\hat{\Phi}(x_1,x_2)\defeq \int_{\A^2} \Phi(u_1,u_2) \psi(u_1x_1+u_2x_2)\,\mathrm{d}u_1\,\mathrm{d}u_2\]
with respect to the character $\psi$ (defined in Section \ref{section: Notations}). As before, we construct the corresponding principal series vector
\[\hat{f}_{s}(g)\defeq f_{\hat{\Phi},s}(g)\in \mathcal{I}(s).\]
Finally, define the Eisenstein series attached to $f_s$ by
\[\mathrm{Eis}(f_s)(g)\defeq \sum_{\gamma\in \N(F)\mathrm{A}(F)\bs \G(F)} f_s(\gamma g)\]
which is absolutely convergent for $\mathrm{Re}(s)>1$ and has a meromorphic continuation to all of $\C$ with at most simple poles at $0$ and $1$. It has a constant term given by 
\[f_s(g)+\tilde{f}_s(g),\]
where
\begin{equation}\label{eq: def second constant term of Eis}
    \tilde{f}_s(g)\defeq \hat{f}_{1-s}\left(w(g^t)^{-1}\right)\in \mathcal{I}(1-s),
\end{equation}
and $w\defeq\big(\begin{smallmatrix}
    & 1\\
    1& 
\end{smallmatrix}\big)$.
\subsection{Whittaker model}\label{section: Whittaker}
For any non-trivial unitary character $\psi$ of $N$ define the space of Whittaker functions
\[\mathcal{W}(\G,\psi)\defeq \{W\in C^{\infty}(\G)\; \text{of moderate growth}\;\lvert\; W(ng)=\psi(n)W(g)\; \forall n\in N\;\forall g\in \G\}. \]
We say an automorphic representation $\pi$ is generic if it injects into $\mathcal{W}(\G(\A),\psi)$. If that is the case, we call $\mathcal{W}(\pi,\psi)$ the Whittaker model of $\pi$. Similarly, we call a local representation $\pi_v$ generic if it injects into $\mathcal{W}(\G(F_v),\psi)$ and denote the image by $\mathcal{W}(\pi_v,\psi_v)$. Recall that any automorphic representation $\pi$ is decomposable into local representations $\bigotimes_v \pi_v$, and if $\pi$ is generic, then so is $\pi_v$ for any $v$. When $F_v$ is non-archimedean, $\psi_v$ is unramified and $W_{\pi_v}\in \mathcal{W}(\pi_v,\psi_v)$ is the arithmetically normalised newvector \emph{i.e.} satisfying $W_{\pi_v}(1)=1$, then we have an explicit formula \cite[Theorem 4.1.]{Miyauchi}
\begin{equation}\label{eq: Miyauchi}
    W_{\pi_v} \left(\diag{\varpi_v^n}\right)=\begin{cases}
    N(\p_v)^{-n/2}\;\,\frac{\alpha_1^{n+1}-\alpha_2^{n+1}}{\alpha_1-\alpha_2},\quad & \text{if $n\geq 0$,}\\
    0,& \text{otherwise,}
\end{cases} 
\end{equation}
where $\alpha_1,\alpha_2$ are the Satake parameters of $\pi_v$, as discussed in Section \ref{section: Notations}. If $\psi_v $ is ramified we shift the newvector $W_{\pi_v}$ by $\lambda\in \oo_v$, a generator of the different ideal $\mathfrak{d}_v$, and work instead with $W_{\pi_v}^{\lambda}(g)\defeq W_{\pi_v}\left(\diag{\lambda}g\right).$ 

Finally, if $\pi_v$ is generic and unitary, we have the inner product on $\mathcal{W}(\pi_v,\psi)$ 
\[\langle W_1,W_2\rangle_{\pi_v} \defeq  \int_{F_v^{\times}} W_1\left(\begin{pmatrix}
    y& \\
    &1
\end{pmatrix}\right)\overline{W_2}\left(\begin{pmatrix}
    y& \\
    &1
\end{pmatrix}\right) \mathrm{d}^{\times}y,\]
which is $G$-invariant. We normalise it by
\begin{equation}\label{eq: Whittaker inner product}
    \langle W_1,W_2\rangle\defeq\begin{cases}
 \frac{\zeta_v(2)}{L_v(1,\pi\otimes \tilde{\pi})}\langle W_1,W_2\rangle_{\pi_v}\quad &\text{if $v$ is non-archimedean,}\\
\langle W_1,W_2\rangle_{\pi_v} & \text{if $v$ is archimedean.}
\end{cases}
\end{equation}
\subsection{Rankin--Selberg \texorpdfstring{$L$}{L}--function}\label{section: RankinSelberg}
We give a brief overview of the theory for $\mathrm{GL}(2)\times \mathrm{GL}(2)$ Rankin--Selberg $L$--function; see \cite{CogRS} for more details. Fix an irreducible cuspidal representation $\pi_0$ and let $\vph_0\in\pi_0$. Take any generic representation (cuspidal or Eisenstein) $\pi$ and a vector $\vph\in \pi$. Finally, for $f_s\in \mathcal{I}(s)$ a principal series vector we define the zeta integral
\[\Psi(f_s,\vph_0,\overline{\vph})\defeq \int_{\G(F)\bs \G(\A)} \varphi_0\,\overline{\varphi}\, \mathrm{Eis}(f_s)\]
which is absolutely convergent away from the poles of the Eisenstein series since $\vph_0$ is a cusp form. We will denote by $W_{\vph_0}$ and $W_{\vph}$ the images of $\vph_0$ and $\vph$ in their respective Whittaker models. Suppose that $W_{\vph_0}$ and $W_{\vph}$ factor into products of local Whittaker functions, $W_{\vph_0}=\prod_v W_{\vph_0,v}$, $W_{\vph}=\prod_v W_{\vph,v}$, and that $f_s$  is attached to a factorable Schwartz--Bruhat function $\Phi\in \mathcal{S}(\A^2)$. In this case, the zeta integral is Eulerian, \emph{i.e.} we can factorize it into a product of local integrals, namely
\begin{align*}
\Psi(f_s,\vph_0,\overline{\vph})=\int_{\N(\A)\bs \G(\A)} W_{\varphi_0}(g)\,\overline{W_{\varphi}(g)}\, f_s(g)\; \mathrm{d}g=\prod_v \Psi_v(f_{s,v},W_{\varphi_0,v},\overline{W_{\varphi,v}}).
\end{align*}
where we define
\begin{equation}\label{eq: def of local zeta integral}
\Psi_v(f_{s,v},W_{\varphi_0,v},\overline{W_{\varphi,v}})\defeq \int_{\N(F_v)\bs \G(F_v)} W_{\varphi_0,v}(g)\,\overline{W_{\varphi,v}(g)}\, f_{s,v}(g)\; \mathrm{d}g.
\end{equation}
In the case, $\Phi_v(x_1,x_2)=\textbf{1}_{\oo_v}(x_1)\textbf{1}_{\oo_v}(x_2)$ when $v$ is non-archimedean, or $\Phi_{v}(x_1,x_2)=e^{-\pi(x_1^2+x_2^2)}$ for a real place $v$, or $\Phi_{v}(x_1,x_2)=e^{-2\pi(\lvert x_1\rvert^2+\lvert x_2\rvert^2)}$ for a complex place $v$, and $W_{\pi_{0,v}},W_{\pi_v}$ being the arithmetically normalised newvectors, as discussed in \ref{section: Whittaker}, we have the local Rankin--Selberg $L$ function given by
\begin{equation*}
    L_v(s,\pi_0\otimes \tilde{\pi})=\int_{\N(F_v)\bs \G(F_v)} W_{\pi_{0,v}}(g)\,\overline{W_{\pi_v}(g)}\, f_{s,v}(g)\; \mathrm{d}g. 
\end{equation*}

\subsection{Regularised spectral decomposition}\label{section: RegSpecDecomp}
We follow \cite[Section 4.3]{MVGL2} for the details of this section. Firstly, we discuss the choice of norms for different representations $\pi$. Recall from Section \ref{section: Autreps} the discussion of the space $\widehat{\mathbb{X}}$. For $\pi\in \widehat{\mathbb{X}}$ and $\vph\in \pi$ we let
\[\lVert \vph \rVert_{\pi}^2=\begin{cases}
    \int_{[\G]} \lvert \vph\rvert^2\, \d g; \quad &\pi \text{ is cuspidal,}\\
    \int_{K(\A)} \lvert f\rvert^2(k)\, \d k; \quad &\pi\text{ is Eisenstein, and } \vph=\mathrm{Eis}(f), \\
    1;\quad &\pi \text{ is a Dirichlet character,}
\end{cases}\]
for more careful discussion on the norms see \cite[Section 4.1.7.]{MVGL2}

We start with the standard spectral decomposition for the functions $\Phi_1,\Phi_2\in C_c^{\infty}(\mathbb{X})$ with respect to the inner product on $\mathbb{X}$
\begin{align}\label{eq: specdecomp}
    \langle \Phi_1,\Phi_2 \rangle =\int_{\widehat{\mathbb{X}}}  \sum_{\vph\in \mathcal{B}(\pi)} \frac{1}{\lVert \vph\rVert_{\pi}^2} \langle \Phi_1,\vph \rangle \langle \vph,\Phi_2 \rangle\, \d\mu_{\mathrm{aut}}(\pi) ,
\end{align}
where the definitions of $ \widehat{\mathbb{X}}$ and $\d\mu_{\mathrm{aut}}(\pi)$ are introduced in Section \ref{section: Autreps}, $\mathcal{B}(\pi)$ is an orthogonal basis of factorable vectors of $\pi$, such that for a generic $\pi$, and $\vph=\bigotimes_v \vph_v\in \mathcal{B}(\pi)$, the image of every $\vph_v$ in $\mathcal{W}(\pi_v,\psi_v)$ is arithmetically normalised, see Section \ref{section: Whittaker}.

In the case when $\Phi_1$ is a function of polynomial growth, and $\pi$ is an Eisenstein series
the inner product $\langle \Phi_1,\mathrm{Eis}(f) \rangle$ does not necessarily converge. In such a case, we define the regularised inner product. For a function $\Phi$ on $\mathbb{X}$ of polynomial growth, define the regularised integral of $\Phi$
\[\int^{\mathrm{reg}}_{\mathbb{X}} \Phi\defeq \int_{\mathbb{X}}\Phi-\mathcal{E}\]
where $\mathcal{E}$ is an appropriately chosen Eisenstein series (for more details see \cite[Section 4.3.5.]{MVGL2}). For this to be well defined, we require mild conditions that $\Phi$ does not have a term that grows like a quadratic twist of $\lvert . \rvert$ near the cusp. This makes $\Phi-\mathcal{E}\in L^1(\mathbb{X})$, and the integral is independent of the choice of regularisation. Similarly, we can define the regularised inner product
\[\langle \Phi_1,\Phi_2\rangle_{\mathrm{reg}} \defeq \int^{\mathrm{reg}}_{\mathbb{X}} \Phi_1 \overline{\Phi_2}\]
which satisfies the regularised spectral decomposition
\begin{align}\label{eq: reg spec decomp}
    \langle \Phi_1,\Phi_2 \rangle_{\mathrm{reg}} = \langle \mathcal{E}_1,\Phi_2 \rangle_{\mathrm{reg}}+\langle \Phi_1,\mathcal{E}_2 \rangle_{\mathrm{reg}}+\int_{\widehat{\mathbb{X}}} \sum_{\vph\in \mathcal{B}(\pi)}\frac{1}{\lVert \vph\rVert_{\pi}^2}\langle \Phi_1,\vph \rangle_{\mathrm{reg}} \langle \vph,\Phi_2 \rangle_{\mathrm{reg}}\, \d\mu_{\mathrm{aut}}(\pi)
\end{align}
where $\mathcal{E}_i$, as above, are Eisenstein series appearing in regularised integrals of $\Phi_i$, satisfying $\Phi_i-\mathcal{E}_i\in L^1(\mathbb{X})$.

\subsection{Spectral weight}\label{section: Spectral Weight}
For the spectral weight 
\begin{equation}
    J_\q(\pi)\defeq N(\q) \frac{\zeta_{\q}^3(1)}{\zeta_{\q}(2)} \prod_{v\mid \q} J(f_{1/2,v}(.x_v), W_{0,v}, \pi_{v})
\end{equation}
we define the renormalised version as in equation \eqref{eq: def of spec weight}. We renormalise it for convenience
\[J_{0,\q}(\pi)\defeq  J_\q(\pi) \frac{\zeta_\q(2)}{L_\q(1,\pi_0\otimes \tilde{\pi}_0)} \lvert L_\q(1/2,\tilde{\pi}\otimes\pi_0)\rvert^2,\]
which satisfies 
\[J_\q(\pi)\asymp J_{0,\q}(\pi)\]
with the implied constant independent of $\pi$, $\pi_0$ and $\q$. We write $\widehat{\G(F_v)}$ for the tempered unitary dual of $\G(F_v)$ and we equip it with a local Plancherel measure $\mathrm{d}\mu_{\mathrm{loc}}$ which is compatible with the Haar measure on $\G(F_v)$, see \cite[Section 4.13.2]{BrumMil}. We have
\begin{align*}
    \prod_{v\mid \q} \int_{\widehat{\G(F_v)}}& J_{0,\q}(\pi) \,\mathrm{d}\mu_{\mathrm{loc}}(\pi_v)= \\
    &\mathrm{vol}^{-1}\, (K_\q) \prod_{v\mid \q} \frac{\zeta_v(2)}{L_v(1,\pi_0\otimes \tilde{\pi}_0)}\int_{\widehat{\G(F_v)}} \sum_{W_v\in \mathcal{B}(\pi_v)} \Big\lvert \int_{\N(F_v)\bs \G(F_v)} f_{1/2,v}(.x_v) W_{0,v} \overline{W}_v \Big \lvert^2 \,\mathrm{d} \mu_{\mathrm{loc}}(\pi_v).
\end{align*} 
By the Whittaker-Plancherel formula (see \cite[Chapter 15]{Wallach}) we have for each $v\mid \q$
\[\int_{\widehat{\G(F_v)}} \sum_{W_v\in \mathcal{B}(\pi_v)} \Big\lvert \int_{\N(F_v)\bs \G(F_v)} f_{1/2,v}(.x_v) W_{0,v} \overline{W}_v \Big \lvert^2 \,\mathrm{d} \mu_{\mathrm{loc}}(\pi_v)=\int_{\N(F_v)\bs \G(F_v)} \lvert f_{1/2,v}(.x_v)\rvert^2 \lvert W_{0,v}\rvert^2\]
and using Proposition \ref{Prop: Local Zeta integrals} (i) for $s_1=s_2=0$ we get
\[\int_{\N(F_v)\bs \G(F_v)} \lvert f_{1/2,v}(.x_v)\rvert^2 \lvert W_{0,v}\rvert^2=\frac{L_v(1,\pi_0\otimes\tilde{\pi}_0)}{\zeta_v(2)}.\]
Combining the above calculations, we obtain
\[\prod_{v\mid \q} \int_{\widehat{\G(F_v)}} J_\q(\pi) \,\mathrm{d}\mu_{\mathrm{loc}}(\pi_v)\asymp \prod_{v\mid \q} \int_{\widehat{\G(F_v)}} J_{0,\q}(\pi) \,\mathrm{d}\mu_{\mathrm{loc}}(\pi_v)= \mathrm{vol}^{-1}\, (K_\q).\]
\subsection{Main Theorem}
We are now ready to state the main theorem of this paper
\begin{theorem}\label{Theorem}
    Suppose that $\pi_0\in \hat{\mathbb{X}}_{\mathrm{gen}}^{\mathrm{unr}}$ is a fixed cuspidal unitary representation that is unramified at all places. For each unitary automorphic representation $\pi\in \hat{\mathbb{X}}_{\mathrm{gen}}$ and an integral ideal $\q\subset F$ coprime to the different ideal of $F$, there exists a non-negative weight function $J_{\q}(\pi)$ (defined in \eqref{eq: def of spec weight}) such that
    \begin{itemize}
        \item $J_{\q}(\pi)=0$ if $\mathfrak{C}(\pi)\nmid \q$,\vspace{0.3pc}
        \item $J_{\q}(\pi)\gg_{\varepsilon} (1-\varepsilon)^{\omega_F(\q)},$ if $\mathfrak{C}(\pi)\mid \q$
        \item $\prod_{v\mid \q} \int_{\widehat{\G(F_v)}} J_\q(\pi) \,\mathrm{d}\mu_{\mathrm{loc}}(\pi_v)\asymp \mathrm{vol}^{-1}\, (K_\q)$ \vspace{0.3pc}
    \end{itemize}
where $\omega_F(\q)$ as in \eqref{eq: ideal omega function} is the prime ideal omega function. Moreover, we have an asymptotic expansion
\begin{align*}
    \int_{\hat{\mathbb{X}}_{\mathrm{gen}} } &\frac{\lvert \Lambda (1/2,\tilde{\pi}\otimes \pi_0)\rvert^2}{\ell(\pi)} J_\q (\pi) \, \mathrm{d}\mu_{\mathrm{aut}}(\pi)=\mathrm{vol}^{\,-1}(K_\q)\Bigg(\frac{\xi^*(1)^3 N(\mathfrak{d})\Lambda(1,\pi_0,\mathrm{Ad})}{3\xi(2)}  \log^3 N(\q)+\\
    &+A_{1,\q} \log^2N(\q)
  +A_{2,\q} \log N(\q)+A_{3,\q}+\mathcal{O}_{F,\varepsilon,\pi_0}\left(N(\q)^{-1/2+\theta+\varepsilon}\right) \Bigg),
\end{align*}
where $\ell(\pi)$ is independent of $\q$ (defined in Section \ref{section: Calculations on the spectral side}), $\xi^*(1)$ is the residue of $\xi$ at $1$, $\theta$ is the spectral gap for $\G$ over $F$, $\mathrm{vol}^{\,-1}(K_\q)=N(\q)\frac{\zeta_{\q}(1)}{\zeta_{\q}(2)}$, and each $A_{j,\q}$ is an explicit $\mathfrak{q}$-dependent term satisfying $A_{j,\q}\ll_{F,\pi_0} \min \{ \omega_F(\q)^{j+2},(\log \log \mathfrak{r})^{j+2}\}$.
\end{theorem}
\begin{remark}\label{remark: similarity to KMV}
    In fact, if $\q=\p^n$, where $\p$ is a prime ideal with either $N(\p)\to \infty$ or $n\to \infty$, then by Remark \ref{remark: coefficients Ajq for prime or depth}, we get that $A_{j,\q}$ are $\q$-independent constants. In this case, we can rewrite the asymptotic expansion as
    \[\int_{\hat{\mathbb{X}}_{\mathrm{gen}} } \frac{\lvert \Lambda (1/2,\tilde{\pi}\otimes \pi_0)\rvert^2}{\ell(\pi)} J_\q (\pi) \, \mathrm{d}\mu_{\mathrm{aut}}(\pi)=\mathrm{vol}^{\,-1}(K_\q)\left(P(\log N(\q))+\mathcal{O}_{F,\varepsilon,\pi_0}\left(N(\q)^{-1/2+\theta+\varepsilon}\right)\right),\]
    where $P$ is a cubic polynomial with leading coefficient $\frac{\xi^*(1)^3 N(\mathfrak{d})\Lambda(1,\pi_0,\text{Ad})}{3\xi(2)}$. This is a very similar result to  \cite[Proposition 1.6.]{KowMichVan} and in our case we improve the error term from $N(\q)^{-1/12+\varepsilon}$ to $N(\q)^{-1/2+\theta+\varepsilon}$.
\end{remark}
\begin{remark}
    Assuming the Generalised Ramanujan Conjecture for $\GL(2)$ over $F$, we can take $\theta=0$ and obtain the square root cancellation and an error term of the form $\mathcal{O}_{F,\varepsilon,\pi_0}\big(N(\q)^{-1/2+\varepsilon}\big)$.
\end{remark}
\begin{remark}
    In \cite{BlomHar}, the authors study a similar problem in the archimedean aspect; however, their main term does not resemble the one mentioned in Theorem \ref{Theorem} because their moment $\mathcal{I}(T, K)$ is not varying over all forms in the conductor family.
\end{remark}
\section{Proof of the Main Theorem}
\subsection{Set-up}\label{section: Set-up}
We start by fixing a cuspidal representation $\pi_0\in \widehat{\mathbb{X}}_{\mathrm{gen}}^{\mathrm{unr}}$ that is unramified at all places, as in the statement of the Theorem \ref{Theorem}. Recall from Section \ref{section: Notations} the choice of an additive character $\psi$ on $F\bs \A$. Let $\vph_0\in \pi_0$ be a factorable form whose Whittaker function $W_{\vph_0}=W_0=\bigoplus_{v\leq \infty} W_{0,v}\in \mathcal{W}(\pi_0,\psi)$, is a spherical vector at all places in the Whittaker model of $\pi_0$, such that $W_{0,v}(1)=1$ for all places $v$.

Let $\q=\prod_{v\mid \q} \p_v^{r_v}$ be an integral ideal tending to infinity. Define the matrix $x=(x_{v})_{v}\in \G(\A)$, where
\[x_v=\diag{\varpi_v^{-r_v}}\hspace{2mm} \text{for}\hspace{2mm} v\mid \q; \hspace{2mm} x_v=\diag{\lambda_v^{-1}} \hspace{2mm} \text{for}\hspace{2mm}v\mid \mathfrak{d}; \hspace{2mm} x_{v}=1 \hspace{2mm} \text{for}\hspace{2mm}v\nmid \q\mathfrak{d},\] where $\varpi_v$, and $\lambda_v$ are as in Section \ref{section: Notations}.

Finally, we make a choice of the Schwartz--Bruhat function $\Phi\in \mathcal{S}(\A^2)$ given locally by
\begin{itemize}
     \item $\Phi_v(x_1,x_2)=\mathbf{1}_{\mathfrak{o}_{v}}(x_1)\mathbf{1}_{\mathfrak{o}_{v}}(x_2)$ for $v<\infty$ and $v\nmid \q \mathfrak{d}$,
     \item $\Phi_{v}(x_1,x_2)=\mathbf{1}_{\mathfrak{o}_{v}}(x_1) \mathbf{1}_{\mathfrak{o}_{v}^{\times}}(x_2)$ for $v\mid \q$,
     \item $\Phi_{v}(x_1,x_2)=e^{-\pi(x_1^2+x_2^2)}$ for a real archimedean place $v$,
     \item $\Phi_{v}(x_1,x_2)=e^{-2\pi(\lvert x_1\rvert^2+\lvert x_2\rvert^2)}$ for a complex archimedean place $v$,
     \item $\Phi_v(x_1,x_2)=\mathbf{1}_{\mathfrak{o}_{v}}(\lambda_v x_1)\mathbf{1}_{\mathfrak{o}_{v}}(x_2)$ for $v\mid \mathfrak{d}$,
\end{itemize}
and take the corresponding induced vector $f_s\in \mathcal{I}(s)$ as defined in \eqref{eq: def of princ series vec}.
We take a Fourier transform of $\Phi$ with respect to the additive character $\psi$, as defined in Section \ref{section: Eis}. We then have a factorisation $\hat{\Phi}= \prod_v \hat{\Phi}_v$, where
 \begin{itemize}
     \item $\hat{\Phi}_v(x_1,x_2)=\mathbf{1}_{\mathfrak{o}_{v}}(x_1)\mathbf{1}_{\mathfrak{o}_{v}}(x_2)$ for $v<\infty$ and $v\nmid \q \mathfrak{d}$
     \item $\hat{\Phi}_{v}(x_1,x_2)=\mathbf{1}_{\mathfrak{o}_{v}}(x_1)( \mathbf{1}_{\mathfrak{o}_{v}}(x_2)-N(\p_v)^{-1} \,\mathbf{1}_{\mathfrak{o}_{v}}(\p_v x_2))$ for $v\mid \q$
     \item $\widehat{\Phi}_{v}(x_1,x_2)=e^{-\pi(x_1^2+x_2^2)}$ for a real archimedean place $v$,
     \item $\widehat{\Phi}_{v}(x_1,x_2)=e^{-2\pi(\lvert x_1\rvert^2+\lvert x_2\rvert^2)}$ for a complex archimedean place $v$,
     \item $\hat{\Phi}_v(x_1,x_2)=\mathbf{1}_{\mathfrak{o}_{v}}(x_1)\mathbf{1}_{\mathfrak{o}_{v}}(\lambda_v x_2)$ for $v\mid \mathfrak{d}$.
 \end{itemize}
We start with the equality of the inner products
\[\langle \vph_0 \mathrm{Eis}(f_{1/2})(.\, x), \vph_0 \mathrm{Eis}(f_{1/2})(.\, x) \rangle=\langle \lvert \mathrm{Eis}(f_{1/2})\rvert^2(.\, x),  \lvert\vph_0\rvert^2\rangle,\]
where, because $\vph_0$ is a cuspidal vector, both of the inner products above are absolutely convergent.
\subsection{Calculations on the spectral side}\label{section: Calculations on the spectral side}
Let $\Re(s)$ be sufficiently positive.
We use the spectral decomposition \eqref{eq: specdecomp} to write
\begin{equation}\label{eq: spectral side specdecomp}
    \langle \vph_0 \mathrm{Eis}(f_{s})(.\, x), \vph_0 \mathrm{Eis}(f_{s})(.\, x) \rangle= \int_{\hat{\mathbb{X}}} \sum_{\vph\in \mathcal{B}(\pi)} \frac{1}{\lVert \vph\rVert_{\pi}^2} \lvert \Psi(f_{s}(.x),\vph_0,\overline{\vph})\rvert^2 \, \mathrm{d}\mu_{\text{aut}}(\pi),
\end{equation}
where there is no contribution from the non-generic spectrum; see \cite[Lemma 4.1]{JanaGLn}, and for a generic $\pi$ we take the sum to run over an arithmetically normalised orthogonal basis of factorable vectors, see Section \ref{section: Whittaker}. The only terms $\vph_v\in \mathcal{B}(\pi_v)$ for $v\nmid \q$ that contribute to the sum are the spherical vectors because both $f_{s,v}$ and $\vph_{0,v}$ are spherical at these places. By Schur's Lemma there is a constant $\ell_\q(\pi)$, so that
\[\lVert \vph\rVert_{\pi}^2=\ell_\q(\pi) \prod_{v\mid \q} \lVert W\rVert^2.\]
But in fact, by the choice of our inner product on $W(\pi_v,\psi_v)$, see Section \ref{section: Whittaker}, we have that $\ell_\q(\pi)=\ell(\pi)$ is independent of the ideal $\q$, where for a cuspidal representation $\ell(\pi)\asymp L(1,\pi,\mathrm{Ad})$; see \cite[Lemma 4.1]{JanaNunesRec} for a more careful discussion. Therefore, the right-hand side of equation \eqref{eq: spectral side specdecomp} transforms to
\[\langle \vph_0 \mathrm{Eis}(f_{s})(.\, x), \vph_0 \mathrm{Eis}(f_{s})(.\, x) \rangle=\int_{\hat{\mathbb{X}}_{\mathrm{gen}}} \frac{1}{\ell(\pi)} \sum_{W_{\vph}\in \mathcal{B}(\pi)} \lvert \Psi(f_{s}(.x),W_0,\overline{W_{\vph}})\rvert^2 \, \mathrm{d}\mu_{\text{aut}}(\pi),\]
where the sum runs over the orthonormal basis with respect to the inner product as defined in \eqref{eq: Whittaker inner product}. Because of the choice of factorable vectors, the sum decomposes as
\[\sum_{W_{\vph}\in \mathcal{B}(\pi)} \lvert \Psi(f_{s}(.x),W_0,\overline{W_{\vph}})\rvert^2=\prod_v \sum_{W_{\vph,v}\in \mathcal{B}(\pi_v)} \lvert \Psi(f_{s,v}(.x_v),W_{0,v},\overline{W_{\vph,v}})\rvert^2.\]

For places $v\nmid \q$, the corresponding local factors can be simplified significantly. In this case, $f_{s,v}$ and $W_{0,v}$ are spherical vectors, so the only non-zero contribution comes from representations $\pi_v$ that are unramified. Then we have
\[ \sum_{W_{\vph,v}\in \mathcal{B}(\pi_v)} \lvert \Psi(f_{s,v}(.x_v),W_{0,v},\overline{W_{\vph,v}})\rvert^2=\lvert \Psi(f_{s,v},W_{0,v},\overline{W_{\pi_v}})\rvert^2 =\lvert \Lambda_v(s,\pi_0\otimes \tilde{\pi})|^2,\]
where $W_{\pi_v}$ is the arithmetically normalised spherical vector of $\pi_v$, as in \ref{section: Whittaker}, and we get an $L$--function as in Section \ref{section: RankinSelberg}. For $v\mid \q$ we define the sum
\[J(f_{s,v}(.x_v), W_{0,v}, \pi_{v})\defeq \frac{1}{\lvert L_v(s,\tilde{\pi}\otimes \pi_0)\rvert^2} \sum_{W_{v}\in \mathcal{B}(\pi_v)}\lvert \Psi_v(f_{s,v}(.x_v), W_{0,v},\overline{W_v})\rvert^2.\]
We combine the normalised factors for $v\mid \q$, take $s=1/2$ and define the local spectral weight as
\begin{equation}\label{eq: def of spec weight}
    J_\q(\pi)\defeq N(\q) \frac{\zeta_{\q}^3(1)}{\zeta_{\q}(2)} \prod_{v\mid \q} J(f_{1/2,v}(.x_v), W_{0,v}, \pi_{v}).
\end{equation}
We call the product above the local spectral weight. We will show in Proposition \ref{Prop: Lower bound on specwght} that the local spectral weight satisfies 
\[J_\q(\pi)\gg_{\varepsilon} (1-\varepsilon)^{\omega_F(\q)} ,\]
for any $\varepsilon>0$ where $\omega_F(\q)$ is defined in \eqref{eq: ideal omega function}.
 Coming back to the inner product, we arrive at the identity
\[N(\q) \frac{\zeta_{\q}^3(1)}{\zeta_{\q}(2)}\langle \vph_0 \mathrm{Eis}(f_{1/2})(.\, x), \vph_0 \mathrm{Eis}(f_{1/2})(.\, x) \rangle=\int_{\hat{\mathbb{X}}_{\mathrm{gen}}} \frac{\lvert \Lambda(1/2,\pi_0\otimes \tilde{\pi})\rvert^2}{\ell(\pi)}J_\q(\pi)\, \mathrm{d}\mu_{\text{aut}}(\pi),\]
which is the spectral side stated in Theorem \ref{Theorem}.

\subsection{Calculations on the period side}\label{section: Calculations on the period side}
In this section we deal with
\[\langle \lvert \mathrm{Eis}(f_{1/2})\rvert^2(.\, x), \lvert \vph_0\rvert^2\rangle.\]
As before, we want to start with the spectral decomposition. This time, however, the square of the Eisenstein series fails to be in $\mathrm{L}^1(\mathbb{X})$. Because of this, when we try to expand the inner product as 
\[\langle \lvert \mathrm{Eis}(f_{1/2})\rvert^2(.\, x), \lvert \vph_0\rvert^2\rangle=\int_{\pi} \sum_{\vph\in \mathcal{B}(\pi)} \frac{1}{\lVert \vph\rVert_{\pi}^2} \langle \lvert \mathrm{Eis}(f_{1/2})\rvert^2(.\, x), \vph\rangle\langle \vph, \lvert \vph_0\rvert^2\rangle,\]
the expression $\langle \lvert \mathrm{Eis}(f_{1/2})\rvert^2(.\, x), \vph\rangle$ is not defined when $\vph$ is an Eisenstein series. Because of that, we want to use the regularised version of the spectral decomposition, see \cite[eq. 4.20]{MVGL2}.
Before that, we want to deform the Eisenstein series away from $s=1/2$ as below.
\[\overline{\mathrm{Eis}(f_{1/2+s_1})}\mathrm{Eis}(f_{1/2+s_2})\]
for some non-zero $s_1,s_2$ with sufficiently small real parts satisfying $\overline{s_1}\pm s_2\neq 0$. We define the \emph{degenerate term} to be the Eisenstein series formed out of the constant terms of the product above, namely,
\begin{equation}\label{eq: Eis series out of constant terms}
    \mathcal{E}_{s_1,s_2}\defeq \mathrm{Eis}(\overline{f_{1/2+s_1}}f_{1/2+s_2}+\overline{\Tilde{f}_{1/2+s_1}}f_{1/2+s_2}+\overline{f_{1/2+s_1}}\Tilde{f}_{1/2+s_2}+\overline{\tilde{f}_{1/2+s_1}}\tilde{f}_{1/2+s_2});
\end{equation}
see \eqref{eq: def second constant term of Eis} for the definition of $\tilde{f}.$ The regularisation makes
\[\overline{\mathrm{Eis}(f_{1/2+s_1})}\mathrm{Eis}(f_{1/2+s_2})-\mathcal{E}_{s_1,s_2}\]
integrable as discussed in Section \ref{section: RegSpecDecomp}, and we obtain the regularised spectral decomposition
\begin{equation}\label{eq: period side reg spec decomp}
\begin{split}
    \langle \overline{\mathrm{Eis}(f_{1/2+s_1})}\mathrm{Eis}(f_{1/2+s_2})(.x),&\lvert \vph_0\rvert^2 \rangle =\langle \mathcal{E}_{s_1,s_2}(.\,x),\lvert \vph_0\rvert^2 \rangle+\\
   &+\int_{\widehat{\mathbb{X}}} \sum_{\vph\in \mathcal{B}(\pi)}\frac{1}{\lVert \vph\rVert_{\pi}^2}\langle \overline{\mathrm{Eis}(f_{1/2+s_1})}\mathrm{Eis}(f_{1/2+s_2})(.x),\vph \rangle_{\mathrm{reg}} \langle \vph,\lvert \vph_0\rvert^2 \rangle\, \d\mu_{\mathrm{aut}}(\pi),
\end{split}
\end{equation}
where the sums are over arithmetically normalised orthogonal basis of $\mathcal{B}(\pi)$. The calculations of the degenerate term are done in Section \ref{section: Degenerate term}, and for the regularised term see Section \ref{section: Regularised term}.
\subsection{Proof of the Theorem}
\begin{proof}[Proof of Theorem \ref{Theorem}]
We start by writing
\begin{equation}\label{eq: normalised inner product identity}
    N(\q)\frac{\zeta_\q^3(1)}{\zeta_\q(2)}\langle \vph_0 \mathrm{Eis}(f_{s})(.\, x), \vph_0 \mathrm{Eis}(f_{s})(.\, x) \rangle=N(\q)\frac{\zeta_\q^3(1)}{\zeta_\q(2)}\langle \lvert \mathrm{Eis}(f_{1/2})(.\, x)\rvert^2, \lvert \vph_0\rvert^2\rangle,
\end{equation}
where, by the calculations in Section \ref{section: Calculations on the spectral side}, the left-hand side evaluates to 
\[N(\q)\frac{\zeta_\q^3(1)}{\zeta_\q(2)}\langle \vph_0 \mathrm{Eis}(f_{s})(.\, x), \vph_0 \mathrm{Eis}(f_{s})(.\, x) \rangle=\int_{\hat{\mathbb{X}}_{\mathrm{gen}}}\frac{\lvert \Lambda(s,\tilde{\pi}\otimes \pi_0)\rvert^2}{\ell(\pi)}J_\q(\pi)\, \mathrm{d}\mu_{\text{aut}}(\pi).\]
The first property of $J_\q(\pi)$ in Theorem \ref{Theorem} is satisfied by Proposition \ref{Prop: Local Specwght}, the second property follows from Proposition \ref{Prop: Lower bound on specwght}, and finally, we obtain the third property in the Section $\ref{section: Spectral Weight}$. For the right-hand side of \eqref{eq: normalised inner product identity}, we use the regularised spectral decomposition; see \ref{eq: period side reg spec decomp} 
\begin{align*}
    \langle \overline{\mathrm{Eis}(f_{1/2+s_1})}(.x)\mathrm{Eis}&(f_{1/2+s_2})(.x),\lvert \vph_0\rvert^2 \rangle =\langle \mathcal{E}_{s_1,s_2}(.\,x),\lvert \vph_0\rvert^2 \rangle\\
    & +\int_{\hat{\mathbb{X}}}\sum_{\vph\in \mathcal{B}(\pi)}\frac{1}{\lVert \vph\rVert_{\pi}^2}\langle \overline{\mathrm{Eis}(f_{1/2+s_1})}(.x)\mathrm{Eis}(f_{1/2+s_2})(.x),\vph \rangle_{\mathrm{reg}} \langle \vph,\lvert \vph_0\rvert^2 \rangle\, \mathrm{d}\mu_{\mathrm{aut}}(\pi).
 \end{align*}
In Section \ref{section: Regularised term} we discuss that the non-generic spectrum does not contribute. Moreover, Proposition \ref{Prop: Degenerate term} gives the asymptotic behavior of the main term
\begin{align*}
    N(\q)\frac{\zeta_\q^3(1)}{\zeta_\q(2)}\lim_{s_1,s_2\to 0} \langle \mathcal{E}_{s_1,s_2}(.\,x),\lvert \vph_0\rvert^2 \rangle&= \mathrm{vol}^{\,-1}(K_\q)\Bigg(\frac{\xi^*(1)^3 N(\mathfrak{d})\Lambda(1,\pi_0,\text{Ad})}{3\xi(2)}  \log^3 N(\q)+\\
    &+A_{1,\q} \log^2N(\q)
  +A_{2,\q} \log N(\q)+A_{3,\q}\Bigg).
\end{align*}
Finally, by Proposition \ref{Prop: Regularised term}, the integral over the generic spectrum is
\[\lim_{s_1,s_2\to 0}\int_{\hat{\mathbb{X}}_{\mathrm{gen}}}\sum_{\vph\in \mathcal{B}(\pi)} \frac{1}{\lVert \vph \rVert_{\pi}^2}\langle \overline{\mathrm{Eis}(f_{1/2+s_1})}(.\,x)\mathrm{Eis}(f_{1/2+s_2})(.x),\vph \rangle_{\mathrm{reg}} \langle \vph,\lvert \vph_0\rvert^2 \rangle\, \mathrm{d}\mu_{\mathrm{aut}}(\pi)\ll_{\varepsilon,\pi_0} N(\q)^{-1/2+\theta+\varepsilon}.\]
Combining the estimates, we obtain the final result.
\end{proof}

\section{Degenerate term}\label{section: Degenerate term}
In this section we consider the degenerate term of the period side in equation \eqref{eq: period side reg spec decomp} and show the following result.
\begin{proposition}\label{Prop: Degenerate term}
We have
\begin{align*}
  N(\q) \frac{\zeta_{\q}^3(1)}{\zeta_{\q}(2)}\lim_{s_1,s_2\to 0} \langle \mathcal{E}_{s_1,s_2}(.\,x),\lvert \vph_0\rvert^2 \rangle=& \mathrm{vol}^{\,-1}(K_\q)\Big( \xi^*(1)^3 N(\mathfrak{d})\frac{\Lambda(1,\pi_0,\text{Ad})}{\xi(2)}  \log^3 N(\q)+A_{1,\q} \log^2N(\q)+\\
  +&A_{2,\q} \log N(\q)+A_{3,\q}\Big),  
\end{align*}
    where each $A_{j,\q}$ is an explicit $\q$-dependent term satisfying $A_{j,\q}\ll_{F,\pi_0} \min \{ \omega_F(\q)^{j+2},(\log \log \mathfrak{r})^{j+2}\}$.
\end{proposition}
For the remainder of this section we will take $s_1,s_2$ such that $\overline{s_1},s_2,\overline{s_1} \pm s_2\neq 0$. To prove the proposition above, we start with the standard Rankin--Selberg computation, transferring from the global to the local problem. We have
\begin{equation}\label{eq: decomposing degenerate term}
    \begin{split}
        \langle \mathcal{E}_{s_1,s_2}(.\,x), \lvert\vph_0\rvert^2 \rangle &=\prod_v \Psi_v(\overline{f_{1/2+s_1,v}}(.x_v)f_{1/2+s_2,v}(.x_v),W_{0,v},\overline{W_{0,v}})\\
    &+\prod_v \Psi_v(\overline{\Tilde{f}_{1/2+s_1,v}}(.x_v)f_{1/2+s_2,v}(.x_v),W_{0,v},\overline{W_{0,v}})\\
    &+\prod_v \Psi_v(\overline{f_{1/2+s_1,v}}(.x_v)\Tilde{f}_{1/2+s_2,v}(.x_v),W_{0,v},\overline{W_{0,v}})\\
    &+\prod_v \Psi_v(\overline{\tilde{f}_{1/2+s_1,v}}(.x_v)\tilde{f}_{1/2+s_2,v}(.x_v),W_{0,v},\overline{W_{0,v}})
    \end{split}
\end{equation}
where $\Psi_v$ is the local Zeta integral as defined in equation \eqref{eq: def of local zeta integral}.
We introduce the functions
\begin{equation}\label{eq: def of G function}
    G(z,w)\defeq \frac{N(\q)^{z+w}N(\mathfrak{d})^{1+2z+2w}}{\xi(2+2z+2w)}\xi(1+2z)\xi(1+2w)\Lambda(1+z+w,\pi_0\otimes \tilde{\pi}_0),
\end{equation}
\[h_1(z,w)\defeq \prod_{v\mid \q} \zeta_v^{-1}(1+2z)\zeta_v^{-1}(1+2w),\]
\[h_2(z,w)\defeq \prod_{v\mid \q} \zeta_v^{-1}(1)\zeta_v^{-1}(1+2w),\]
\[h_3(z,w)\defeq \prod_{v\mid \q} \zeta_v^{-1}(1+2z)\zeta_v^{-1}(1),\]
\[h_4(z,w)\defeq \prod_{v\mid \q} \zeta_v^{-1}(1-2z)\zeta_v^{-1}(1-2w)\frac{\zeta_v(1-2z-2w)}{\zeta_v(1)},\]
which are meromorphic in $z,w$ and will correspond to the four summands in equation \eqref{eq: decomposing degenerate term}. Note that $h_2$, $h_3$ are, in fact, one-variable functions, but we keep the notation as above, for convenience. We need a result on the Taylor coefficients of $h_j$.
\begin{lemma}\label{Lem: Taylor coefficients of zetas}
    The coefficients of the Taylor expansion
    \[h_j(z,w)=\sum_{m,n\geq 0} a_{m,n}^{(j)} z^m w^n\]
    satisfy
    \[a_{m,n}^{(j)}\ll_{m,n} \min \{ \omega_F(\q)^{m+n},(\log \log N(\mathfrak{r}))^{m+n}\}.\]
\end{lemma}
\begin{proof}
    It is easy to see that
\[\zeta_v^{-1}(1+2z)=\sum_{n\ge 0} c_v(n) z^n, \quad\text{ with } |c_v(0)|< 1,\text{ and }|c_v(n)| \ll_n \frac{\log^n N(\p_v)}{N(\p_v)} \text{ for  }n\ge 1,\]
and similarly
\[\frac{\zeta_v(1-2z-2w)}{\zeta_v(1)}=\sum_{n\ge 0} d_v(n) (z+w)^n, \quad\text{ with } d_v(0)= 1,\text{ and }|d_v(n)| \ll_n \frac{\log^n N(\p_v)}{N(\p_v)} \text{ for  }n\ge 1.\]
Thus the coefficient of $z^n$ of $\prod_{v\mid \q} \zeta_v^{-1}(1+2z)$ equals
\begin{equation*}
\sum_{\substack{n_v\in\Z_{\ge 0}\, :\\ \sum_{v\mid \q} n_v =n}} \prod_{v\mid \q}c_v(n_v) \ll_n \sum_{\substack{n_v\in\Z_{\ge 0}\, :\\  \sum_{v\mid \q} n_v =n}}\;\; \prod_{\substack{v\mid \q\,:\\ n_v>0}}\frac{\log^{n_v} N(\p_v)}{N(\p_v)}\ll_n \sum_{\substack{n_1,\dots n_a\in \Z_{\ge 0}\\ \sum_j n_j=n}} \prod_{j=1}^a \Bigg( \sum_{v\mid \q} \frac{\log^{n_j} N(\p_v)}{N(\p_v)}\Bigg)
\end{equation*}
and the coefficient of $(z+w)^n$ of $\prod_{v\mid \q} \tfrac{\zeta_v(1-2z-2w)}{\zeta_v(1)}$ equals
\begin{equation*}
\sum_{\substack{n_v\in\Z_{\ge 0}\, :\\ \sum_{v\mid \q} n_v =n}} \prod_{v\mid \q}d_v(n_v) \ll_n \sum_{\substack{n_v\in\Z_{\ge 0}\, :\\  \sum_{v\mid \q} n_v =n}}\;\;  \prod_{\substack{v\mid \q\,:\\ n_v>0}}\frac{\log^{n_v} N(\p_v)}{N(\p_v)}\ll_n \sum_{\substack{n_1,\dots n_a\in \Z_{\ge 0}\\ \sum_j n_j=n}} \prod_{j=1}^a \Bigg( \sum_{v\mid \q} \frac{\log^{n_j} N(\p_v)}{N(\p_v)}\Bigg).
\end{equation*}
In the case when $\q$ has few distinct prime factors, we can bound
\[\sum_{v\mid \q} \frac{\log^m N(\p_v)}{N(\p_v)}\ll_m \sum_{v\mid \q} 1=\omega_F(\q). \]
On the other hand, we can write
\[\sum_{v\mid \q} \frac{\log^m N(\p_v)}{N(\p_v)}\ll_m \sum_{N(\p_v)\leq \log (N(\mathfrak{r}))} \frac{\log^m N(\p_v)}{N(\p_v)}\ll_m (\log \log (N(\mathfrak{r})))^m\]
which we calculate using summation by parts and Landau prime ideal theorem. Therefore, we get the bound
\[\sum_{\substack{n_v\in\Z_{\ge 0}\, :\\ \sum_{v\mid \q} n_v =n}} \prod_{v\mid \q}c_v(n_v) \ll_n  \min \{ \omega_F(\q)^{n},(\log \log N(\mathfrak{r}))^{n}\}\]
and a similarly
\[\sum_{\substack{n_v\in\Z_{\ge 0}\, :\\ \sum_{v\mid \q} n_v =n}} \prod_{v\mid \q}d_v(n_v) \ll_n  \min \{ \omega_F(\q)^{n},(\log \log N(\mathfrak{r}))^{n}\}.\]
If we now consider Taylor coefficients of $h_j(z,w)$, then by multiplying bounds on the Taylor coefficients of $\prod_{v\mid \q} \zeta_v^{-1}(1+2z)$ and $\prod_{v\mid \q} \zeta_v^{-1}(1+2z)$ we obtain the statement of the Lemma.
\end{proof}
\begin{remark}\label{remark: coeffs of hi for p^n}
    If $\q=\p$ is a prime ideal we get a stronger result that
    \[a_{m,n}^{(j)}(z,w) \begin{cases}
        =&1\quad m=n=0\\
        \ll_{\varepsilon,m,n} &N(\p)^{-(1-\varepsilon)} \quad \mathrm{otherwise}.
    \end{cases}\]
    Similarly, if $\q=\p^k$, where $\p$ is a fixed prime ideal we have
    \[a_{m,n}^{(j)}(z,w)\ll_{\varepsilon,m,n,\p} 1.\]
\end{remark}
\subsection{Local calculations} 
In this section, we prove the following result.
\begin{proposition}\label{Prop: Local Zeta integrals}
    For each place $v$ we have\vspace{4pt}
    \begin{enumerate}
        \item[(i)] $\Psi_v(\overline{f_{1/2+s_1,v}}(.x_v)f_{1/2+s_2,v}(.x_v),W_{0,v},\overline{W_{0,v}})=G_v(\overline{s_1},s_2)h_{1,v}(\overline{s_1},s_2),$\vspace{4pt}
        \item[(ii)] $\Psi_v(\overline{\Tilde{f}_{1/2+s_1,v}}(.x_v)f_{1/2+s_2,v}(.x_v),W_{0,v},\overline{W_{0,v}})= G_v(-\overline{s_1},s_2)h_{2,v}(\overline{s_1},s_2),$\vspace{4pt}
        \item[(iii)] $\Psi_v(\overline{f_{1/2+s_1,v}}(.x_v)\Tilde{f}_{1/2+s_2,v}(.x_v),W_{0,v},\overline{W_{0,v}})= G_v(\overline{s_1},-s_2)h_{3,v}(\overline{s_1},s_2),$\vspace{4pt}
        \item[(iv)] $\Psi_v(\overline{\tilde{f}_{1/2+s_1,v}}(.x_v)\tilde{f}_{1/2+s_2,v}(.x_v),W_{0,v},\overline{W_{0,v}})= G_v(-\overline{s_1},-s_2)h_{4,v}(\overline{s_1},s_2)-$\vspace{4pt}\\
        $\mathbf{1}_{v\mid \q}\cdot G_v(-\overline{s_1},-s_2)h_{4,v}(\overline{s_1},s_2)8\overline{s_1}s_2(\overline{s_1}+s_2)\frac{\zeta_v^3(1) \log^3 N(\p_v)}{N(\p_v)^{r_v+1}}(1+\varepsilon_{\overline{s_1},s_2,v}),$\vspace{4pt}
    \end{enumerate}
where $\lim_{s_1,s_2\to 0} \varepsilon_{\overline{s_1},s_2,v}=0$.
\end{proposition}
We split the proof over several subsections depending on the place $v$. The calculations in Sections \ref{section: Unramified calculations}, and \ref{section: Discriminant calculations} are standard. For the ramified places see Section \ref{section: Ramified calculations}. Before we calculate the zeta integrals, we prove a standard lemma that will be useful for us.
\begin{lemma}\label{Lem: Whittaker integral calc}
    Let $\pi_v$ be an unramified representation of $\mathrm{G}(F_v)$ and $W\in \mathcal{W}(\pi_v,\psi_v)$ be an unramified vector in the Whittaker model of $\pi_v$ such that $W(1)=1$. Then we have an equality
    \[ \int_{F^{\times}_v} \lvert W\rvert^2  \begin{pmatrix}
            y&\\
            &1
        \end{pmatrix} \lvert y\rvert^{s}
        \,\mathrm{d}^{\times}y=\frac{L_v(1+s,\pi\otimes\tilde{\pi})}{\zeta_v(2+2s)}. \]
\end{lemma}
\begin{proof}
    We write $\alpha_1,\alpha_2$ for the Satake parameters of $\pi_v$ as in Section \ref{section: Notations}. Recall the formula \eqref{eq: Miyauchi} for the normalised newvector, which we apply to our integral
    \[\int_{F^{\times}_v} \lvert W\rvert^2  \begin{pmatrix}
            y&\\
            &1
        \end{pmatrix} \lvert y\rvert^{s}
        \,\mathrm{d}^{\times}y=\sum_{n\geq 0}\lvert W\rvert^2  \begin{pmatrix}
            \varpi_v^n&\\
            &1
        \end{pmatrix}\lvert \varpi_v \rvert^{ns}=\sum_{n\geq 0}N(\p_v)^{-n(1+s)}\left\lvert \frac{\alpha_1^{n+1}-\alpha_2^{n+1}}{\alpha_1-\alpha_2} \right\rvert^{2}.\]
After standard geometric series calculations, we arrive at the desired result (see Section \ref{section: Notations}).
\end{proof}

\subsubsection{Unramified place away from $\q\mathfrak{d}$.}\label{section: Unramified calculations} 
In the case of an unramified place $v$, see Section \ref{section: Set-up}, we have $\Phi_v=\hat{\Phi}_v$, $\tilde{f}_{s,v}=f_{1-s,v}$ and $x_v=1$. So each of the four local factors is equal to
\[\Psi_v(\overline{f_{1/2\pm s_1,v}}f_{1/2\pm s_2,v},W_{0,v},\overline{W_{0,v}}).\]
Moreover $\overline{f_{1/2\pm s_1,v}}f_{1/2\pm s_2,v}$ is an unramified vector in $\mathcal{I}_v(1\pm \overline{s_1}\pm s_2)$. Hence, it is a constant multiple of $f_{1+\overline{s_1}+s_2,v}$. For $\Re(s)>0$ we have
\[f_{s,v}(1)=\int_{F_v^{\times}} \Phi_{v}(0,\,t) \lvert t\rvert^{2s} \, \mathrm{d}^{\times} t=\xi_v(2s),\]
so by comparing the values at the identity, we obtain the equality
\[\overline{f_{1/2+s_1,v}}f_{1/2+s_2,v}=\frac{\xi_v(1+2\overline{s_1})\xi_v(1+2s_2)}{\xi_v(2+2\overline{s_1}+2s_2)}f_{1+\overline{s_1}+s_2,v}.\]
Therefore, with the notation as in equation \eqref{eq: def of G function}, we get that the four terms are 
   \[\Psi_v(\overline{f_{1/2\pm s_1,v}}f_{1/2\pm s_2,v},W_{0,v},\overline{W_{0,v}})=G_v(\pm \overline{s_1},\pm s_2),\]
and since in this case $h_{j,v}=1$ we get the result of Proposition \ref{Prop: Local Zeta integrals} for these places.
\subsubsection{Place dividing the different ideal}\label{section: Discriminant calculations}
Suppose $v\mid \mathfrak{d}$, for more details see \cite{JanaNunesMom}. In this case, the vectors are just shifted unramified vectors and the local zeta integrals satisfy
\begin{align*}
    \Psi_v(\overline{f_{1/2\pm s_1,v}}f_{1/2\pm s_2,v},W_{0,v},\overline{W_{0,v}})&=N(\mathfrak{d}_v)^{1\pm 2\overline{s_1}\pm 2s_2}L_v(1+\overline{s_1}\pm  s_2, \pi_0\otimes \tilde{\pi_0}) \frac{\zeta_v(1\pm 2\overline{s_1})\zeta_v(1\pm 2s_2)}{\zeta_v(2\pm 2\overline{s_1}\pm 2s_2)}\\
    &=G_v(\pm\overline{s_1},\pm s_2).
\end{align*}

\subsubsection{Ramified place}\label{section: Ramified calculations}
Fix a finite place $v\mid \q$ and write $X=\varpi_v^{-r_v}$. We will use the Bruhat decomposition \eqref{eq: Bruhat decomp} for zeta integrals, and so we need to investigate the behaviour of the induced vectors on the lower triangular matrices. Recall the calculations of $\Phi_v$ and $\hat{\Phi}_v$ from Section \ref{section: Set-up}. 
We have
\begin{equation}\label{eq: calc of princ series vec}
\begin{split}
    f_{s,v}\begin{pmatrix}
            y&\\
            c&1
        \end{pmatrix}&=\lvert y\rvert^s f_{s,v}\begin{pmatrix}
            1&\\
            c&1
        \end{pmatrix}=\lvert y\rvert^s \int_{F_v^{\times}} \Phi(tc,\, t) \lvert t\rvert^{2s} \, \mathrm{d}^{\times} t=\lvert y\rvert^s \int_{F_v^{\times}} \mathbf{1}_{\mathfrak{o}_{v}}(tc) \mathbf{1}_{\mathfrak{o}_{v}^{\times}}(t) \lvert t\rvert^{2s} \, \mathrm{d}^{\times} t\\
        &=\lvert y\rvert^s \int_{\mathfrak{o}_v^{\times}} \mathbf{1}_{\mathfrak{o}_{v}}(tc) \, \mathrm{d}^{\times} t=\lvert y\rvert^s \mathbf{1}_{\mathfrak{o}_{v}}(c) \int_{\mathfrak{o}_v^{\times}} \, \mathrm{d}^{\times} t=\lvert y\rvert^s \mathbf{1}_{\mathfrak{o}_{v}}(c),
\end{split}
\end{equation}
and similarly 
\begin{align*}
\tilde{f}_{s,v}\begin{pmatrix}
            y&\\
            c&1
        \end{pmatrix}&=\lvert y\rvert^{1-s}\hat{f}_{1-s,v}\begin{pmatrix}
            &1\\
            1&-c
        \end{pmatrix}=\lvert y\rvert^{1-s} \int_{\oo_v \setminus \{0\}} ( \mathbf{1}_{\mathfrak{o}_{v}}(tc)-N(\p_v)^{-1} \,\mathbf{1}_{\mathfrak{o}_{v}}(\varpi_v tc))  \lvert t\rvert^{2(1-s)} \, \mathrm{d}^{\times} t.
\end{align*}
In the case when $c\in \mathfrak{o}_v$, the integrand is equal to $\zeta_v^{-1}(1)\lvert t\rvert^{2(1-s)}$, hence $\tilde{f}_{s,v}\big(\begin{smallmatrix}
            y&\\
            c&1
        \end{smallmatrix}\big) = \lvert y\rvert^{1-s}\frac{\zeta_v(2(1-s))}{\zeta_v(1)}.$ If $c\not\in \mathfrak{o}_v$ we make a substitution $t\to ct$ to arrive at
    \begin{align*}
        &\tilde{f}_{s,v}\begin{pmatrix}
            y&\\
            c&1
        \end{pmatrix}=\lvert y\rvert^{1-s}\lvert c\rvert^{-2(1-s)} \int_{c\oo_v \setminus \{0\}} ( \mathbf{1}_{\mathfrak{o}_{v}}(t)-N(\p_v)^{-1} \,\mathbf{1}_{\mathfrak{o}_{v}}(\varpi_v t))  \lvert t\rvert^{2(1-s)} \, \mathrm{d}^{\times} t\\
        =&\lvert y\rvert^{1-s}\lvert c\rvert^{-2(1-s)}\left(\int_{\oo_v \setminus \{0\}} \lvert t\rvert^{2(1-s)} \, \mathrm{d}^{\times} t-N(\p_v)^{-1}\int_{\varpi_v^{-1}\oo_v \setminus \{0\}} \lvert t\rvert^{2(1-s)} \, \mathrm{d}^{\times} t\right)=\lvert y\rvert^{1-s}\frac{\zeta_v(2(1-s))}{\zeta_v(2s-1)} \lvert c\rvert^{-2(1-s)}.
    \end{align*}
Therefore we have
\begin{equation}\label{eq: calc of opposite princ series vec }
    \begin{split}
        \tilde{f}_{s,v}\begin{pmatrix}
            y&\\
            c&1
        \end{pmatrix}=\lvert y\rvert^{1-s}\begin{cases}
            \frac{\zeta_v(2(1-s))}{\zeta_v(1)}\, \quad\,\, &\text{if}\; c\in \mathfrak{o}_v, \\
            \frac{\zeta_v(2(1-s))}{\zeta_v(2s-1)} \lvert c\rvert^{-2(1-s)}\quad &\text{otherwise}.\\
            \end{cases}
    \end{split}
\end{equation}
\begin{proof}[Proof of Proposition \ref{Prop: Local Zeta integrals} (i), (ii), (iii) ]
We do the calculations for the second part, and the other ones follow similarly. We drop the subscript $v$ for simplicity, as all of the calculations are local, and switch to Bruhat coordinates to get that
\[\Psi(\overline{f_{1/2+s_1}}(.\,x)\tilde{f}_{1/2+s_2}(.x),W_{0},\overline{W_{0}} )=\int_{F^{\times}}\int_{F} \lvert W_0\rvert^2  \begin{pmatrix}
            y&\\
            c&1
        \end{pmatrix} \overline{f_{1/2+s_1}}\begin{pmatrix}
            yX&\\
            cX&1
        \end{pmatrix}\tilde{f}_{1/2+s_2}\begin{pmatrix}
            yX&\\
            cX&1
        \end{pmatrix}
        \, \mathrm{d}c\,\frac{\mathrm{d}^{\times}y}{\lvert y\rvert}.\]
Firstly, we make a substitution $c\mapsto c/X$ to arrive at
    \begin{align*}
\Psi(\overline{f_{1/2+s_1}}(.x)\tilde{f}_{1/2+s_2}(.x),W_{0},\overline{W_{0}})=\lvert X\rvert^{\overline{s_1}-s_2}&\int_{F} \overline{f_{1/2+s_1}}\begin{pmatrix}
            1&\\
            c&1
        \end{pmatrix}\tilde{f}_{1/2+s_2}\begin{pmatrix}
            1&\\
            c&1
        \end{pmatrix}\\
        &\int_{F^{\times}} \lvert W_0\rvert^2  \begin{pmatrix}
            y&\\
            c/X&1
        \end{pmatrix} \lvert y\rvert^{\overline{s_1}-s_2}
        \, \mathrm{d}^{\times}y\, \mathrm{d}c.
\end{align*}
From equation \eqref{eq: calc of princ series vec} we read off that $f_{1/2+s_1}\big(\begin{smallmatrix}
            1& \\
            c&1
        \end{smallmatrix}\big)=\mathbf{1}_{\mathfrak{o}}(c)$, so for $c\in \oo$ we have
        \[\lvert W_0\rvert^2  \begin{pmatrix}
            y&\\
            c/X&1
        \end{pmatrix}=\lvert W_0\rvert^2\begin{pmatrix}
            y&\\
            &1
        \end{pmatrix}\begin{pmatrix}
            1&\\
            c/X&1
        \end{pmatrix}=\lvert W_0\rvert^2  \begin{pmatrix}
            y&\\
            &1
        \end{pmatrix}.\]
Also in the same range for $c\in \oo$, from equation \eqref{eq: calc of opposite princ series vec } we have $\tilde{f}_{1/2+s_2}\big(\begin{smallmatrix}
            1& \\
            c&1
        \end{smallmatrix}\big)=\frac{\zeta(1-2s_2)}{\zeta(1)}$ hence
\begin{align*}
    \Psi(\overline{f_{1/2+s_1}}(.x)\tilde{f}_{1/2+s_2}(.x),W_{0},\overline{W_{0}})&=\lvert X\rvert^{\overline{s_1}-s_2}\frac{\zeta(1-2s_2)}{\zeta(1)}\int_{F^{\times}} \lvert W_0\rvert^2  \begin{pmatrix}
            y&\\
            &1
        \end{pmatrix}\lvert y\rvert^{\overline{s_1}-s_2}
        \, \mathrm{d}^{\times}y\\
        &=G(\overline{s_1},-s_2)h_{3}(\overline{s_1},s_2),
\end{align*}
where the last equality follows from applying Lemma \ref{Lem: Whittaker integral calc} with $s=\overline{s_1}-s_2$. 
\end{proof}
Finally, we have a more complicated last term
\begin{proof}[Proof of Proposition \ref{Prop: Local Zeta integrals} (iv)]
Recall from \ref{section: Ramified calculations} that we denote $X=\varpi_v^{-r_v}$.
    As before, we drop the subscript $v$ and write
    \begin{align*}
        \Psi(\overline{\tilde{f}_{1/2+s_1}} (.x)\tilde{f}_{1/2+s_2} (.x),W_{0},\overline{W_{0}} )=\lvert X\rvert^{-\overline{s_1}-s_2}\int_{F} &\overline{\tilde{f}_{1/2+s_1}}\begin{pmatrix}
            1&\\
            c&1
        \end{pmatrix}\tilde{f}_{1/2+s_2}\begin{pmatrix}
            1&\\
            c&1
        \end{pmatrix}\\
        &\int_{F^{\times}} \lvert W_0\rvert^2  \begin{pmatrix}
            y&\\
            c/X&1
        \end{pmatrix} \lvert y\rvert^{-\overline{s_1}-s_2}
        \, \mathrm{d}^{\times}y\, \mathrm{d}c.
    \end{align*}
This time the domain of  $\overline{\tilde{f}_{1/2+s_1}}\big(\begin{smallmatrix}
            1& \\
            c&1
        \end{smallmatrix}\big)\tilde{f}_{1/2+s_2}\big(\begin{smallmatrix}
            1& \\
            c&1
        \end{smallmatrix}\big)$ is not contained in $\mathfrak{o}$ so we cannot guarantee that $\big(\begin{smallmatrix}
            1&\\
            c/X&1
\end{smallmatrix}\big)\in \text{GL}_2(\mathfrak{o}).$ 
We split the integral as
\begin{equation}\label{eq: calc of ram zeta int 1}
    \begin{split}
        &\lvert X\rvert^{-\overline{s_1}-s_2}\int_{\p_v^{-r_v}} \overline{\tilde{f}_{1/2+s_1}}\begin{pmatrix}
            1&\\
            c&1
        \end{pmatrix}\tilde{f}_{1/2+s_2}\begin{pmatrix}
            1&\\
            c&1
        \end{pmatrix}\int_{F^{\times}_v} \lvert W_0\rvert^2  \begin{pmatrix}
            y&\\
            c/X&1
        \end{pmatrix} \lvert y\rvert^{-\overline{s_1}-s_2}
        \, \mathrm{d}^{\times}y\, \mathrm{d}c\\
        +&\lvert X\rvert^{-\overline{s_1}-s_2}\int_{F_v\bs \p_v^{-r_v}} \overline{\tilde{f}_{1/2+s_1}}\begin{pmatrix}
            1&\\
            c&1
        \end{pmatrix}\tilde{f}_{1/2+s_2}\begin{pmatrix}
            1&\\
            c&1
        \end{pmatrix}\int_{F^{\times}_v} \lvert W_0\rvert^2  \begin{pmatrix}
            y&\\
            c/X&1
        \end{pmatrix} \lvert y\rvert^{-\overline{s_1}-s_2}
        \, \mathrm{d}^{\times}y\, \mathrm{d}c.
    \end{split}
\end{equation}
For the first summand in equation \eqref{eq: calc of ram zeta int 1}, we use $G(\oo)$ invariance of the Whittaker function, and the integrals separate into a product
\begin{equation}\label{eq: calc of ram zeta int 2}
    \lvert X\rvert^{-\overline{s_1}-s_2}\int_{\p^{-r}} \overline{\tilde{f}_{1/2+s_1}}\begin{pmatrix}
            1&\\
            c&1
        \end{pmatrix}\tilde{f}_{1/2+s_2}\begin{pmatrix}
            1&\\
            c&1
        \end{pmatrix}\, \mathrm{d}c \int_{F^{\times}} \lvert W_0\rvert^2  \begin{pmatrix}
            y&\\
            &1
        \end{pmatrix} \lvert y\rvert^{-\overline{s_1}-s_2}
        \, \mathrm{d}^{\times}y,
\end{equation}
where as before we apply Lemma \ref{Lem: Whittaker integral calc} with $s=-\overline{s_1}-s_2$ and obtain that the integral over $y$ in equation \eqref{eq: calc of ram zeta int 2} is $\frac{L(1+\overline{s_1}+s_2,\pi_0\otimes \tilde{\pi}_0)}{\zeta(2+2\overline{s_1}+s_2)}.$
The integral over $y$ above is $\frac{L_v(1+\overline{s_1}+s_2,\pi_0\otimes \tilde{\pi}_0)}{\zeta_v(2+2\overline{s_1}+s_2)}.$
For the integral over $c$ in equation \eqref{eq: calc of ram zeta int 2}, we write it as
\[\int_{F} \overline{\tilde{f}_{1/2+s_1}}\begin{pmatrix}
            1&\\
            c&1
        \end{pmatrix}\tilde{f}_{1/2+s_2}\begin{pmatrix}
            1&\\
            c&1
        \end{pmatrix}\, \mathrm{d}c-\int_{F\bs \p^{-r}} \overline{\tilde{f}_{1/2+s_1}}\begin{pmatrix}
            1&\\
            c&1
        \end{pmatrix}\tilde{f}_{1/2+s_2}\begin{pmatrix}
            1&\\
            c&1
        \end{pmatrix}\, \mathrm{d}c,\]
where for the second integral we recall from equation \eqref{eq: calc of opposite princ series vec } that for the range of $c\not\in \oo$ we have $\tilde{f}_{s}\big(\begin{smallmatrix}
            1& \\
            c&1
        \end{smallmatrix}\big)=\frac{\zeta(2(1-s))}{\zeta(2s-1)} \lvert c\rvert^{-2(1-s)}$.
Therefore we obtain
\begin{equation}\label{eq: calc of integral of ftilde}
\begin{split}
    \int_{F\bs \p^{-r}} \overline{\tilde{f}_{1/2+s_1}}\begin{pmatrix}
            1&\\
            c&1
        \end{pmatrix}\tilde{f}_{1/2+s_2}\begin{pmatrix}
            1&\\
            c&1
        \end{pmatrix}\, \mathrm{d}c&=\frac{\zeta(1-2\overline{s_1})\zeta(1-2s_2)}{\zeta(2\overline{s_1})\zeta(2s_2)}\int_{F\bs \p^{-r_v}} \lvert c\rvert^{2(\overline{s_1}+s_2-1)} \, \mathrm{d}c\\
        &=\frac{\zeta(1-2\overline{s_1})\zeta(1-2s_2)\zeta(1-2\overline{s_1}-2s_2)}{\zeta(1)\zeta(2\overline{s_1})\zeta(2s_2)} N(\p)^{-(r+1)(1-2\overline{s_1}-2s_2)}.
\end{split}
\end{equation}
We split the first integral into two ranges
\begin{align*}
    \int_{F} \overline{\tilde{f}_{1/2+s_1}}\begin{pmatrix}
            1&\\
            c&1
        \end{pmatrix}\tilde{f}_{1/2+s_2}\begin{pmatrix}
            1&\\
            c&1
        \end{pmatrix}\, \mathrm{d}c&=\int_{\oo} \overline{\tilde{f}_{1/2+s_1}}\begin{pmatrix}
            1&\\
            c&1
        \end{pmatrix}\tilde{f}_{1/2+s_2}\begin{pmatrix}
            1&\\
            c&1
        \end{pmatrix}\, \mathrm{d}c+\\
        &+\int_{F\setminus \oo} \overline{\tilde{f}_{1/2+s_1}}\begin{pmatrix}
            1&\\
            c&1
        \end{pmatrix}\tilde{f}_{1/2+s_2}\begin{pmatrix}
            1&\\
            c&1
        \end{pmatrix}\, \mathrm{d}c
\end{align*}
and using equation \eqref{eq: calc of opposite princ series vec } as before we calculate
\[\int_{\oo} \overline{\tilde{f}_{1/2+s_1}}\begin{pmatrix}
            1&\\
            c&1
        \end{pmatrix}\tilde{f}_{1/2+s_2}\begin{pmatrix}
            1&\\
            c&1
        \end{pmatrix}\, \mathrm{d}c=\frac{\zeta(1-2\overline{s_1})\zeta(1-2s_2)}{\zeta^2(1)},\]
and
\[\int_{F\setminus \oo} \overline{\tilde{f}_{1/2+s_1}}\begin{pmatrix}
            1&\\
            c&1
        \end{pmatrix}\tilde{f}_{1/2+s_2}\begin{pmatrix}
            1&\\
            c&1
        \end{pmatrix}\, \mathrm{d}c=\frac{\zeta(1-2\overline{s_1})\zeta(1-2s_2)\zeta(1-2\overline{s_1}-2s_2)}{\zeta(1)\zeta(2\overline{s_1})\zeta(2s_2)} N(\p)^{-(1-2\overline{s_1}-2s_2)}.\]
Therefore, combining the two, we get
\begin{align*}
    \int_{F} \overline{\tilde{f}_{1/2+s_1}}\begin{pmatrix}
            1&\\
            c&1
        \end{pmatrix}\tilde{f}_{1/2+s_2}\begin{pmatrix}
            1&\\
            c&1
        \end{pmatrix}\, \mathrm{d}c=&\frac{\zeta(1-2\overline{s_1})\zeta(1-2s_2)\zeta(1-2\overline{s_1}-2s_2)}{\zeta(1)}\cdot\\
        \Big( \left(1-N(\p)^{-1}\right)\left(1-N(\p)^{-(1-2\overline{s_1}-2s_2)}\right)
        +& N(\p)^{-(1-2\overline{s_1}-2s_2)}\left(1-N(\p)^{-2\overline{s_1}}\right)\left(1-N(\p)^{-2s_2}\right)\Big)\\
        =&\frac{\zeta(1-2\overline{s_1}-2s_2)}{\zeta(1)}.
\end{align*}
Therefore, after recalling the definitions in equation \eqref{eq: def of G function}, the first summand in equation \eqref{eq: calc of ram zeta int 1} is
\begin{equation}\label{eq: calc of ram zeta int 3}
G(-\overline{s_1},-s_2)h_{4}(\overline{s_1},s_2)\left(1-\frac{\zeta(1-2\overline{s_1})\zeta(1-2s_2)}{\zeta(2\overline{s_1})\zeta(2s_2)}N(\p)^{-(r+1)(1-2\overline{s_1}-2s_2)}\right).
\end{equation}

For the remaining term in equation \eqref{eq: def of G function} we write the Iwasawa decomposition 
\[\begin{pmatrix}
    y&\\
    \alpha&1
\end{pmatrix}=\begin{pmatrix}
    \alpha&\\
    &\alpha
\end{pmatrix}\begin{pmatrix}
    1& y/\alpha\\
    &1
\end{pmatrix}\begin{pmatrix}
    y/\alpha^{2}&\\
    &1
\end{pmatrix}\begin{pmatrix}
    &-1\\
    1&
\end{pmatrix}\begin{pmatrix}
    1&1/\alpha\\
    &1
\end{pmatrix}\]
for $\alpha^{-1}\in \oo$. Now, since $W_0$ is the spherical vector of a representation with the trivial central character, we calculate
\[\big\lvert W_0 \big\rvert^2 \begin{pmatrix}
    y&\\
    \alpha&1
\end{pmatrix}=\lvert \psi\rvert^2\begin{pmatrix}
    1&y/\alpha\\
    &1
\end{pmatrix} \big\lvert W_0 \big\rvert^2\begin{pmatrix}
    y/\alpha^{2}&\\
    &1
\end{pmatrix}=\big\lvert W_0 \big\rvert^2\begin{pmatrix}
    y/\alpha^{2}&\\
    &1
\end{pmatrix}.\]
So for a fixed $c$ we have
\begin{align*}
    \int_{F^{\times}} \lvert W_0\rvert^2  \begin{pmatrix}
            y&\\
            c/X&1
        \end{pmatrix} \lvert y\rvert^{-\overline{s_1}-s_2} \, \mathrm{d}^{\times}y&=\int_{F^{\times}} \lvert W_0\rvert^2  \begin{pmatrix}
            yX^2/c^2&\\
            &1
        \end{pmatrix} \lvert y\rvert^{-\overline{s_1}-s_2}\, \mathrm{d}^{\times}y=\\
 &=\lvert X^2/c^2\rvert^{-\overline{s_1}-s_2}\int_{F^{\times}} \lvert W_0\rvert^2  \begin{pmatrix}
            y&\\
            &1
        \end{pmatrix} \lvert y\rvert^{-\overline{s_1}-s_2}\, \mathrm{d}^{\times}y=\\
 &=\lvert c\rvert^{2(\overline{s_1}+s_2)}
 )\lvert X\rvert^{-2\overline{s_1}-2s_2} \frac{L(1-\overline{s_1}-s_2,\pi\otimes\widetilde{\pi})}{\zeta(2-2\overline{s_1}-2s_2)}.
\end{align*}
Therefore, as in equation \eqref{eq: calc of integral of ftilde}, the outer integral in the second term of equation \eqref{eq: calc of ram zeta int 1} is
\[\int_{F\bs \p^{-r}}\overline{\tilde{f}_{1/2+s_1}}\begin{pmatrix}
            1&\\
            c&1
        \end{pmatrix}\tilde{f}_{1/2+s_2}\begin{pmatrix}
            1&\\
            c&1
        \end{pmatrix}\lvert c\rvert^{2(\overline{s_1}+s_2)}\, \mathrm{d}c=\frac{\zeta(1-2\overline{s_1})\zeta(1-2s_2)}{\zeta(2\overline{s_1})\zeta(2s_2)}N(\p)^{-r-1} \]
and so 
\begin{equation}\label{eq: calc of ram zeta int 4}
    \begin{split}
        &\lvert X\rvert^{-\overline{s_1}-s_2}\int_{F\bs \p^{-r}} \overline{\tilde{f}_{1/2+s_1}}\begin{pmatrix}
            1&\\
            c&1
        \end{pmatrix}\tilde{f}_{1/2+s_2}\begin{pmatrix}
            1&\\
            c&1
        \end{pmatrix}\int_{F^{\times}} \lvert W_0\rvert^2  \begin{pmatrix}
            y&\\
            c/X&1
        \end{pmatrix} \lvert y\rvert^{-\overline{s_1}-s_2}
        \, \mathrm{d}^{\times}y\, \mathrm{d}c\\
        =&\lvert X\rvert^{-\overline{s_1}-s_2}N(\p)^{2r(\overline{s_1}+s_2)} \frac{L(1-\overline{s_1}-s_2,\pi\otimes\widetilde{\pi})}{\zeta(2-2\overline{s_1}-2s_2)}\frac{\zeta(1-2\overline{s_1})\zeta(1-2s_2)}{\zeta(2\overline{s_1})\zeta(2s_2)}N(\p)^{-r-1}\\
        =& G(-\overline{s_1},-s_2)h_{4}(\overline{s_1},s_2)N(\p)^{-r(1-2\overline{s_1}-2s_2)-1}\frac{\zeta(1-2\overline{s_1})\zeta(1-2s_2)\zeta(1)}{\zeta(1-2\overline{s_1}-2s_2)\zeta(2\overline{s_1})\zeta(2s_2)}.
    \end{split}
\end{equation}
We combine both of the terms in equations \eqref{eq: calc of ram zeta int 3} and \eqref{eq: calc of ram zeta int 4} together and they give
\begin{align*}
&\Psi(\overline{\tilde{f}_{1/2+s_1}} (.x)\tilde{f}_{1/2+s_2} (.x),W_{0},\overline{W_{0}} )=   G(-\overline{s_1},-s_2)h_{4}(\overline{s_1},s_2)\Bigg(1-\\
&N(\p)^{-(r+1)(1-2\overline{s_1}-2s_2)}\frac{\zeta(1-2\overline{s_1})\zeta(1-2s_2)\zeta(1)}{\zeta(2\overline{s_1})\zeta(2s_2)}\Big( \zeta^{-1}(1)-N(\p)^{-2\overline{s_1}-2s_2}\zeta^{-1}(1-2\overline{s_1}-2s_2) \Big)\Bigg)\\
&=G(-\overline{s_1},-s_2)h_{4}(\overline{s_1},s_2)\left(1-N(\p)^{-(r+1)(1-2\overline{s_1}-2s_2)}\frac{\zeta(1-2\overline{s_1})\zeta(1-2s_2)\zeta(1)}{\zeta(2\overline{s_1})\zeta(2s_2)\zeta(2\overline{s_1}+2s_2)}\right). 
\end{align*}
Finally, using the Taylor expansion, we write 
\[N(\p)^{-(r+1)(1-2\overline{s_1}-2s_2)}\frac{\zeta(1-2\overline{s_1})\zeta(1-2s_2)\zeta(1)}{\zeta(2\overline{s_1})\zeta(2s_2)\zeta(2\overline{s_1}+2s_2)}=8\overline{s_1}s_2(\overline{s_1}+s_2)\frac{\zeta^3(1) \log^3 N(\p)}{N(\p)^{r+1}}(1+\varepsilon_{\overline{s_1},s_2}),\]
where $\lim_{s_1,s_2\to 0} \varepsilon_{\overline{s_1},s_2}=0.$
\end{proof}

\subsection{Global conclusion}
We combine the results from the previous section and write $z,w$ instead of $\overline{s_1},s_2$. We take the product of the local zeta integrals in the equation \eqref{eq: decomposing degenerate term} over all places. We use Proposition \ref{Prop: Local Zeta integrals} for sufficiently large $\Re(z),\Re(w)$, to write the first three terms in equation \eqref{eq: decomposing degenerate term} as
\[G(z,w)h_1(z,w)+G(-z,w)h_2(z,w)+G(z,-w)h_3(z,w)\]
and for the fourth term, we get
\[ \Psi(\overline{\tilde{f}_{1/2+\overline{z}}}(.x)\tilde{f}_{1/2+w}(.x),W_{0},\overline{W_{0}})=G(-z,-w)h_4(z,w) \prod_{v\mid \q}\left(1-8zw(z+w)\frac{\zeta_v^3(1) \log^3 N(\p_v)}{N(\p_v)^{r_v+1}}(1+\varepsilon_{z,w,v})\right).\]
Recall from the definition of $G(-z,-w)$ (see equation \eqref{eq: def of G function}) that it has simple poles at $z,w,z+w=0$. Therefore, if we expand the product above, then the only nonzero terms in the limit $\lim_{z,w\to 0}$ will be
\[G(-z,-w)h_4(z,w)\left(1-8zw(z+w)\sum_{v\mid \q}\frac{\zeta_v^3(1) \log^3 N(\p_v)}{N(\p_v)^{r_v+1}}\right).\]
The limit of 
\[-G(-z,-w)h_4(z,w)8zw(z+w)\sum_{v\mid \q}\frac{\zeta_v^3(1) \log^3 N(\p_v)}{N(\p_v)^{r_v+1}}\]
exists, and we consider it separately later. Therefore, we are left with the calculation of
\[\lim_{z,w\to 0} G(z,w)h_1(z,w)+G(-z,w)h_2(z,w)+G(z,-w)h_3(z,w)+G(-z,-w)h_4(z,w)\]
for which we have a general result.
\begin{lemma}\label{Lem: Residue cancellation}
    Suppose $H_1,H_2: \C \to \C$ are holomorphic functions in a punctured neighbourhood of zero with at most simple poles at zero. Define $G(z,w):=H_1(z+w)H_2(z)H_2(w)$. Suppose further that $h_j:\C\times \C \to \C$ for $1\leq j\leq 4$ are holomorphic in both variables and satisfy
    \begin{align*}
        h_1(z,0)&=h_3(z,0), \qquad h_2(z,0)=h_4(z,0),\\
       h_1(0,w)&=h_2(0,w), \qquad h_3(0,w)=h_4(0,w),\\
        h_1(-z,z)&=h_4(-z,z), \qquad  h_2(z,z)=h_3(z,z).
    \end{align*}
Then the limit
\[\lim_{z,w\to 0} G(z,w)h_1(z,w)+G(-z,w)h_2(z,w)+G(z,-w)h_3(z,w)+G(-z,-w)h_4(z,w)\]
exists.
\end{lemma}

\begin{proof}
    We have 4 different singularities to consider: $z,w,z\pm w=0$, and we need to show that all of them are removable. To show that the limit as $z$ tends to 0 exists, we split the expression as
    \begin{multline*}
        H_2(w)\Big(H_2(z)H_1(z+w)h_1(z,w)+H_2(-z)H_1(-z+w)h_2(z,w)\Big)\\
        +H_2(-w)\Big(H_2(z)H_1(z-w)h_3(z,w)+H_2(-z)H_1(-z-w)h_4(z,w)\Big)
    \end{multline*}
and notice that the limits in both of the terms exist since $h_1(0,w)=h_2(0,w)$ and $ h_3(0,w)=h_4(0,w)$. Consideration of other singularities follows similarly.
\end{proof}
This gives us all the tools to prove Proposition \ref{Prop: Degenerate term}.

\begin{proof}[Proof of Proposition \ref{Prop: Degenerate term}]
First, combining \eqref{eq: decomposing degenerate term} and Proposition \ref{Prop: Local Zeta integrals} we write
\begin{align*}
    &\langle \mathcal{E}_{\overline{z},w}(.\,x), \lvert\vph_0\rvert^2 \rangle=G(z,w)h_1(z,w)+G(-z,w)h_2(z,w)+\\
+&G(z,-w)h_3(z,w)+G(-z,-w)h_4(z,w)\left(1-8zw(z+w)\sum_{v\mid \q}\frac{\zeta_v^3(1) \log^3 N(\p_v)}{N(\p_v)^{r_v+1}}\right).
\end{align*}
We have \begin{align*}
    \lim_{z,w \to 0} &G(-z,-w)h_4(z,w)8zw(z+w)\sum_{v\mid \q}\frac{\zeta_v^3(1) \log^3 N(\p_v)}{N(\p_v)^{r_v+1}}=\\
    &-\frac{2c^3 \Lambda(1,\pi_0,\text{Ad})N(\mathfrak{d})}{\xi(2)\zeta_{\q}(1)} \sum_{v\mid \q} \frac{\zeta_v^3(1)\log^3 N(\p_v)}{N(\p_v)^{r_v+1}}\ll 1.
\end{align*}
Letting
\[H_1(z)=\frac{N(\q)^{z}N(\mathfrak{d})^{1+2z}}{\xi(2+2z)}\Lambda(1+z,\pi_0\otimes \tilde{\pi}_0)\]
and
\[H_2(z)=\xi(1+2z),\]
we apply Lemma \ref{Lem: Residue cancellation} with $G$ as in the statement. So, the limit 
\[\lim_{z,w\to 0} \left(G(z,w)h_1(z,w)+G(-z,w)h_2(z,w)+G(z,-w)h_3(z,w)+G(-z,-w)h_4(z,w)\right)\]
exists. Now we calculate the limit as a function of $N(\q)$. We write a Laurent expansion of
\[G(z,w)=\frac{N(\q)^{z+w}N(\mathfrak{d})^{1+2z+2w}}{\xi(2+2z+2w)}\xi(1+2z)\xi(1+2w)\Lambda(1+z+w,\pi_0\otimes \tilde{\pi}_0)\]
around the origin and group the terms according to the power of $\log N(\q)$
\[\frac{P_0(z,w)}{zw(z+w)}+\log N(\q) \frac{P_1(z,w)}{zw(z+w)}+\log^2 N(\q) \frac{P_2(z,w)}{zw(z+w)}+\log^3 N(\q) \frac{P_3(z,w)}{zw(z+w)}+\dots, \]
where $P_j$ are power series in $z$ and $w$ with coefficients independent of $\q$ and no coefficients of degree $< j$. Note that because the limit 
\[\lim_{z,w\to 0} \left(G(z,w)h_1(z,w)+G(-z,w)h_2(z,w)+G(z,-w)h_3(z,w)+G(-z,-w)h_4(z,w)\right)\]
exists, we can ignore the terms with positive degree in $z$ and $w$. So we can omit 
\[\log^4 N(\q)\frac{P_4(z,w)}{zw(z+w)}+\log^5 N(\q) \frac{P_5(z,w)}{zw(z+w)}+\log^6 N(\q) \frac{P_6(z,w)}{zw(z+w)}+\dots, \]
and instead consider
\[\frac{P_0(z,w)}{zw(z+w)}+\log N(\q) \frac{P_1(z,w)}{zw(z+w)}+\log^2 N(\q) \frac{P_2(z,w)}{zw(z+w)}+\log^3 N(\q) \frac{P_3(z,w)}{zw(z+w)}, \]
where now $P_j$ are polynomials in $z$ and $w$ of combined degree $\leq 3$ with coefficients independent of $\q$ and $P_3$ is a homogeneous polynomial of degree $3$.

For $h_j$ we know by Lemma \ref{Lem: Taylor coefficients of zetas} that the coefficients of the terms of degree $m$ are bounded by\\
$\min \{ \omega_F(\q)^m,(\log \log \mathfrak{r})^m\}$. Therefore, combining the Taylor expansions for $G$ and $h_j$ we see that the main $\q$-dependent term is $\log^3 N(\q)$. For example the main term contribution of $G_1(z,w)h_1(z,w)$ is
\begin{equation*}
    \frac{\log^3 N(\q)}{6}\frac{(z+w)^3}{2z 2w(z+w)}\frac{\xi^{*}(1)^2 \Lambda^* (1,\pi_0\otimes\tilde{\pi_0})}{\xi(2)}N(\mathfrak{d})h_1(0,0).
\end{equation*}
Therefore, combining the 4 main terms gives 
\begin{align*}
    \log^3 N(\q)\frac{\xi^{*}(1)^2 \Lambda^* (1,\pi_0\otimes\tilde{\pi_0})}{6\xi(2)}N(\mathfrak{d})h_1(0,0)\left( \frac{(z+w)^2}{4zw}-\frac{(z-w)^2}{4zw}-\frac{(z-w)^2}{4zw}+\frac{(z+w)^2}{4zw}\right)\\
    =\log^3 N(\q)\frac{\xi^{*}(1)^3 \Lambda (1,\pi_0,\mathrm{Ad})}{3\xi(2)}N(\mathfrak{d})\zeta_{\q}^{-2}(1).
\end{align*}
Hence, the whole limit becomes
\begin{align*}
    \lim_{z,w\to 0}\langle \mathcal{E}_{\overline{z},w}(.\,x), \lvert\vph_0\rvert^2 \rangle=&N(\mathfrak{d})\frac{\xi^*(1)^3 \Lambda(1,\pi_0,\text{Ad})}{3\xi(2)\zeta_{\q}^2(1)}\log^3 N(\q)+\\
    +&A'_{1,\q} \log^2 N(\q)+A'_{2,\q} \log N(\q)+A'_{3,\q},
\end{align*}
where $A'_{m,\q}\ll_{F,\pi_0} \min \{ \omega_F(\q)^m,(\log \log \mathfrak{r})^m\}$. After we multiply both sides by $N(\q) \frac{\zeta_{\q}^3(1)}{\zeta_{\q}(2)}$ and use the estimate $\zeta_\q(1)\ll_F \min \{ \omega_F(\q),\log \log \mathfrak{r}\}$ we get that
\begin{align*}
  N(\q) \frac{\zeta_{\q}^3(1)}{\zeta_{\q}(2)}\lim_{s_1,s_2\to 0} \langle \mathcal{E}_{s_1,s_2}(.\,x),\lvert \vph_0\rvert^2 \rangle=& \mathrm{vol}^{\,-1}(K_\q)\Big( \xi^*(1)^3 N(\mathfrak{d})\frac{\Lambda(1,\pi_0,\text{Ad})}{\xi(2)}  \log^3 N(\q)+\\
  +&A_{1,\q} \log^2N(\q)+A_{2,\q} \log N(\q)+A_{3,\q}\Big),  
\end{align*}
    where $A_{j,\q}$ are explicit $\q$-dependent terms satisfying $A_{j,\q}\ll_{F,\pi_0} \min \{ \omega_F(\q)^{j+2},(\log \log \mathfrak{r})^{j+2}\}$.
\end{proof}
\begin{remark}\label{remark: coefficients Ajq for prime or depth}
    In fact, if $\q=\p^n$ is a power of a prime ideal, where $N(\p)\to \infty$ or $n\to \infty$, by using Remark \ref{remark: coeffs of hi for p^n} and following the proof of Proposition \ref{Prop: Degenerate term} we can take $A_{j,\q}$ to be $\q$-independent constants.
\end{remark}
\section{Regularised term}\label{section: Regularised term}
In this section, we study the rest of the terms on the period side of the equation \eqref{eq: period side reg spec decomp} and we show that all of these terms contribute to the error term.
\begin{proposition}\label{Prop: Regularised term}
    We have for any $\varepsilon>0$
    \begin{align*}
     \lim_{s_1,s_2\to 0}\int_{\hat{\mathbb{X}}}\sum_{\vph\in \mathcal{B}(\pi)}\frac{1}{\lVert \vph\rVert^2}\langle \overline{\mathrm{Eis}(f_{1/2+s_1})}(.\,x)\mathrm{Eis}(f_{1/2+s_2})(.x),\vph \rangle_{\mathrm{reg}} \langle \vph,\lvert \vph_0\rvert^2 \rangle\, \mathrm{d}\mu_{\mathrm{aut}}(\pi)\ll_{\varepsilon,\pi_0} N(\q)^{-1/2+\theta+\varepsilon},
    \end{align*}
    where $\langle,\rangle_{\mathrm{reg}}$ is defined in section \ref{section: RegSpecDecomp}, and $\theta$ is as in \ref{section: Notations}.
\end{proposition}
From Proposition \ref{Prop: Degenerate term}, we know that the limit as $s_1,s_2\to 0$ of the degenerate term exists, therefore, it does so for the regularised term as well. So for the rest of this section, we take $s_2=0$, $s_1\neq 0$, and $\Re(s_1)>-1/2$. This way, all of the upcoming regularised inner products are well defined as discussed in Section \ref{section: RegSpecDecomp}. 

For a generic representation $\pi$, we take an orthogonal basis $\mathcal{B}(\pi)$ consisting of factorable vectors. Recall that the vector $\lvert \vph_0\rvert^2$ is $\mathrm{K}(F_v)$-invariant at all places $v$. Therefore for any $\vph\in \mathcal{B}(\pi)$ (or $\vph$ a Hecke character) we have 
\[\langle \vph,\lvert \vph_0\rvert^2 \rangle=0,\]
unless $\vph$ is $\mathrm{K}(F_v)$ invariant for all places $v$. In fact, if $\pi$ is a trivial representation, then the contribution is
\begin{align*}
    \langle \overline{\mathrm{Eis}(f_{1/2+s_1})}(.\,x)\mathrm{Eis}(f_{1/2})(.x),1\rangle_{\mathrm{reg}}\lVert \vph_0\rVert_{\pi}^2=\langle \mathrm{Eis}(f_{1/2}),\mathrm{Eis}(f_{1/2+s_1})\rangle_{\mathrm{reg}}\lVert \vph_0\rVert_{\pi}^2,
\end{align*}
and the regularised inner product of two Eisenstein series is equal to $0$ \cite[Lemma 3.1]{Han}.

So we only need to consider the generic spectrum for which we have the remaining integral
\begin{equation}\label{eq: regularised generic spectrum}
\int_{\widehat{\mathbb{X}}_{\mathrm{gen}}^{\mathrm{unr}}}\frac{1}{\lVert \vph_{\pi}\rVert_{\pi}^2}\langle \overline{\mathrm{Eis}(f_{1/2+s_1})}(.x)\,\mathrm{Eis}(f_{1/2})(.x),\vph_{\pi} \rangle_{\mathrm{reg}} \langle \vph_{\pi},\lvert \vph_0\rvert^2 \rangle\, \mathrm{d}\mu_{\mathrm{aut}}(\pi)
\end{equation}
where $\widehat{\mathbb{X}}_{\mathrm{gen}}^{\mathrm{unr}}$ is defined in Section \ref{section: Autreps} and $\vph_{\pi}\in \pi$ is the vector such that the image of $\vph_{\pi}$ in the Whittaker model of $\pi$, $W_{\vph_{\pi}}=\prod_v W_{\vph_{\pi},v}$ is factorable , and $W_{\vph_{\pi},v}$ is the spherical vector satisfying $W_{\vph_{\pi},v}(1)=1$. 

The final step before proving the main proposition of this section is to study the integrand of the integral in equation \eqref{eq: regularised generic spectrum}.
\begin{lemma}\label{Lem: Regularised triple product calc}
    Let $\pi\in \widehat{\mathbb{X}}_{\mathrm{gen}}^{\mathrm{unr}}$ and $\vph_{\pi}\in\pi$ be the normalised vector unramified at all places. For any $\varepsilon,N>0$ we have
    \begin{itemize}
        \item $ \langle\overline{\mathrm{Eis}(f_{1/2+s_1})}(.x)\mathrm{Eis}(f_{1/2})(.x),\vph_{\pi} \rangle_{\mathrm{reg}}\ll_{\varepsilon,F}  \Lambda^2(1/2+\overline{s_1},\tilde{\pi})N(\q)^{-1/2+\theta+\varepsilon}$, if $\pi$ is cuspidal, \vspace{0.4pc}
        \item $\langle\overline{\mathrm{Eis}(f_{1/2+s_1})}(.x)\mathrm{Eis}(f_{1/2})(.x),\vph_{\pi} \rangle_{\mathrm{reg}}\ll_{\varepsilon,F} \Lambda^2(1/2+\overline{s_1},\tilde{\pi}) N(\q)^{-1/2+\varepsilon}$, if $\pi$ is Eisenstein,
    \end{itemize}
where $s_1$ varies over a compact neighbourhood of $0$.
\end{lemma}
\begin{proof}
    We combine the arguments for the two parts. We start by taking $s_1$ such that $\Re(s_1)\gg 1$. Then by the regularised Rankin--Selberg argument  \cite[Section 4.4.3.]{MVGL2}  we get
    \[\langle \overline{\mathrm{Eis}(f_{1/2+s_1})}\mathrm{Eis}(f_{1/2}),\vph_{\pi}(.x^{-1}) \rangle_{\text{reg}}=\int_{\N(\A)\bs \G(\A)} \overline{W_{\vph_{\pi}}}(.x^{-1})W_{f_{1/2}} \overline{f_{1/2+s_1}},\]
    for any generic $\pi$, where $W_{\vph_{\pi}}, W_{f_{1/2}}$ are the corresponding Whittaker functions, and the integral on the right converges absolutely. Since $\pi_v$ is unramified at all places $v$ we can write the integral as
    \[\int_{\N(\A)\bs \G(\A)} \overline{W_{\vph_{\pi}}}(.x^{-1})W_{f_{1/2}} \overline{f_{1/2+s_1}}=L^{\q\mathfrak{d}}(1/2+\overline{s_1},\mathcal{I}(1/2)\times \tilde{\pi}) \Psi_{\q\mathfrak{d}}, \]
    where $\Psi_{\q\mathfrak{d}}$ is a product of local zeta integrals at places $v\mid \q\mathfrak{d}$. Since the Rankin--Selberg $L$--function admits a meromorphic continuation, we can make sense of the expression on the right-hand side above for any $s_1$. In fact, this is the value of the regularised inner product \emph{i.e.}
    \[\langle \overline{\mathrm{Eis}(f_{1/2+s_1})}\mathrm{Eis}(f_{1/2}),\vph_{\pi}(.x^{-1}) \rangle_{\text{reg}}=L^{\q\mathfrak{d}}(1/2+\overline{s_1},\mathcal{I}(1/2)\times \tilde{\pi}) \Psi_{\q\mathfrak{d}}\]
    for any value of $s_1$, see \cite[Section 4.4.3.]{MVGL2}. The calculations at $v\mid \mathfrak{d}$ give a multiple of a local $L$--function; see \cite[ Remark 4.1.]{JanaNunesMom}.
    
    For the places $v\mid \q$ we do the Bruhat decomposition
    \begin{equation}\label{eq: calc of local zeta in reg term}
    \begin{split}
    \int_{\N(F_v)\bs \G(F_v)} W_{f_{1/2,v}} \overline{W_{\vph,v}(.x_v^{-1})}&\; \overline{f_{1/2+s_1,v}}=\\
    &=\int_{F_v^{\times}}\int_{F_v}  W_{f_{1/2,v}}  \begin{pmatrix}
            y&\\
            c&1
        \end{pmatrix} \overline{W_{\vph,v}}\begin{pmatrix}
            y\varpi_v^{r_v}&\\
            c\varpi_v^{r_v}&1
        \end{pmatrix}\overline{f_{1/2+s_1}}\begin{pmatrix}
            y&\\
            c&1
        \end{pmatrix}
        \, \mathrm{d}c\,\frac{\d^{\times}y}{\lvert y\rvert}
        \\
        &=\int_{F_v^{\times}}\int_{\mathfrak{o}_v} \lvert y\rvert^{1/2+\overline{s_1}} W_{f_{1/2,v}}  \begin{pmatrix}
            1&\\
            c&1
        \end{pmatrix} \overline{W_{\vph,v}}\begin{pmatrix}
            y\varpi_v^{r_v}&\\
            c\varpi_v^{r_v}&1
        \end{pmatrix}
        \, \mathrm{d}c\,\frac{\d^{\times}y}{\lvert y\rvert}\\
        &=\int_{F_v^{\times}}\int_{\mathfrak{o}_v} \lvert y\rvert^{1/2+\overline{s_1}} W_{f_{1/2,v}}  \begin{pmatrix}
            y&\\
            c&1
        \end{pmatrix} \overline{W_{\vph,v}}\begin{pmatrix}
            y\varpi_v^{r_v}&\\
            &1
        \end{pmatrix}
        \, \mathrm{d}c\,\frac{\d^{\times}y}{\lvert y\rvert}.
        \end{split}
    \end{equation} 
In the final equality, we use the fact that $W_{\vph,v}$ is spherical at all places. Moreover for $c\in \oo_v$ we have
\begin{align*}
   W_{f_{1/2},v}  \begin{pmatrix}
            y&\\
            c&1
        \end{pmatrix}=&\lvert y\rvert^{1/2}\int_{F_v}\int_{F_v^{\times}}\Phi_v\left((t\, a)\begin{pmatrix}
            1&\\
            c&1
        \end{pmatrix}\right)\psi_{0,v}(ya/t) \d^{\times}t\, \d a\\ 
        =&\lvert y\rvert^{1/2}\int_{F_v}\int_{F_v^{\times}}\mathbf{1}_{\oo_v}(t+ca)\mathbf{1}_{\oo_v^{\times}}(a)\psi_{0,v}(ya/t) \d^{\times}t\, \mathrm{d}a\\
        =&\lvert y\rvert^{1/2}\int_{a\in \oo_v^{\times}}\int_{t\in \oo_v}\psi_{0,v}(ya/t) \, \d^{\times}t\, \mathrm{d}a\\
        =&\lvert y\rvert^{1/2}\int_{\oo_v}\mathbf{1}_{\oo_v}(y/t)-N(\p_v)^{-1}\mathbf{1}_{\oo_v}(\varpi_v y/t)\,\d^{\times}t \\
        =&
        \begin{cases}
            \lvert y\rvert^{1/2} \left(-N(\p_v)^{-1}+\frac{n+1}{\zeta_v(1)}\right) \quad &\text{if}\;  y\in \p_v^{n}\;\,\mathrm{for}\;\, n\geq-1,  \\
            0 & \mathrm{otherwise},
        \end{cases}
\end{align*}
where we calculated the $a$ integral by splitting it as $\int_{\oo_v}-\int_{\varpi_v \oo_v}$. Thus, the whole integrand in equation \eqref{eq: calc of local zeta in reg term} is independent of $c$, and we get
\begin{align*}
    &\int_{\N(F_v)\bs \G(F_v)} W_{f_{1/2,v}} \overline{W_{\vph,v}(.x^{-1})}\, \, \overline{f_{1/2+s_1,v}}=\int_{F_v^{\times}} \lvert y\rvert^{-1/2+\overline{s_1}} W_{f_{1/2},v}  \begin{pmatrix}
            y&\\
            &1
        \end{pmatrix} \overline{W_{\vph,v}}\begin{pmatrix}
            y\varpi_v^{r_v}&\\
            &1
        \end{pmatrix}
        \, \mathrm{d}^{\times}y=\\
        =&- N(\p_v)^{-1+\overline{s_1}}\, \overline{W_{\vph,v}}\begin{pmatrix}
            \varpi_v^{r-1}&\\
            &1
        \end{pmatrix} +\sum_{n\geq 0}N(\p_v)^{-n\overline{s_1}}\left(-N(\p_v)^{-1}+\frac{n+1}{\zeta_v(1)}\right) \overline{W_{\vph,v}}\begin{pmatrix}
            \varpi_v^{n+r}&\\
            &1
        \end{pmatrix}.
\end{align*}
We calculate the above in terms of Satake parameters $\alpha_1,\alpha_2$ of $\tilde{\pi}_v$, as defined in Section \ref{section: Notations}. We use the formula of Shintani for the normalised newvector \eqref{eq: Miyauchi} and combine the geometric series to write
\begin{align*}
  -N(\p_v)^{-r_v/2}\Bigg[&N(\p_v)^{-1/2+\overline{s_1}} \frac{\alpha_1^{r_v}-\alpha_2^{r_v}}{\alpha_1-\alpha_2}+\frac{N(\p_v)^{-1}}{\alpha_1-\alpha_2}\left( \frac{\alpha_1^{r_v+1}}{1-\alpha_1 N(\p_v)^{-1/2-\overline{s_1}}}-\frac{\alpha_2^{r_v+1}}{1-\alpha_2 N(\p_v)^{-1/2-\overline{s_1}}}\right)\\
  -&\frac{1-N(\p_v)^{-1}}{\alpha_1-\alpha_2}\left(\frac{\alpha_1^{r_v+1}}{(1-\alpha_1 N(\p_v)^{-1/2-\overline{s_1}})^2}-\frac{\alpha_2^{r_v+1}}{(1-\alpha_2 N(\p_v)^{-1/2-\overline{s_1}})^2}\right)\Bigg] .
\end{align*}
By collecting terms with $\alpha_1$ and $\alpha_2$ in the numerators together the expression simplifies to 
\[-\frac{N(\p_v)^{-r_v/2}}{\alpha_1-\alpha_2}\left( \frac{(N(\p_v)^{-1/2+\overline{s_1}}-\alpha_1)\alpha_1^r}{(1-\alpha_1 N(\p_v)^{-1/2-\overline{s_1}})^2}-\frac{(N(\p_v)^{-1/2+\overline{s_1}}-\alpha_2)\alpha_2^r}{(1-\alpha_2 N(\p_v)^{-1/2-\overline{s_1}})^2}\right).\]
Recall that by definition of the parameters $\alpha_i$, see Section \ref{section: Notations}, we have $L_v(s,\tilde{\pi})=(1-\alpha_1 N(\p_v)^s)^{-1}(1-\alpha_2 N(\p_v)^s)^{-1}.$ So by taking out $L_v^2(1/2+\overline{s_1},\tilde{\pi})$ we arrive at
\begin{align*}
    &N(\p_v)^{-r_v/2} L_v^2(1/2+\overline{s_1},\tilde{\pi})\Bigg(\frac{\alpha_1^{r_v+1}-\alpha_2^{r_v+1}}{\alpha_1-\alpha_2}-N(\p_v)^{-1/2}\left(N(\p_v)^{\overline{s_1}}+2N(\p_v)^{-\overline{s_1}}\right)\frac{\alpha_1^{r_v}-\alpha_2^{r_v}}{\alpha_1-\alpha_2}+\\
    &+N(\p_v)^{-1}\left(2+p^{-2\overline{s_1}}\right) \frac{\alpha_1^{r_v-1}-\alpha_2^{r_v-1}}{\alpha_1-\alpha_2}-N(\p_v)^{-3/2-\overline{s_1}}\frac{\alpha_1^{r_v-2}-\alpha_2^{r_v-2}}{\alpha_1-\alpha_2}\Bigg)\\
    &\ll N(\p_v)^{-r_v/2} L_v^2(1/2+\overline{s_1},\tilde{\pi}) (r_v+1) \max \{\lvert \alpha_1^{r_v}\rvert,\lvert \alpha_2^{r_v}\rvert\}, 
\end{align*}
where the inequality follows from the fact that we take $\overline{s_1}$ in the compact neighbourhood of $0$.
Now if $\pi$ is Eisenstein then $\lvert \alpha_1\rvert=\lvert\alpha_2\rvert=1$ and the bound is
\[\int_{\N(F_v)\bs \G(F_v)} W_{f_{1/2,v}} \overline{W_{\vph,v}(.x^{-1})}\, \, \overline{f_{1/2+s_1,v}}\ll_{\varepsilon} N(\p_v)^{r_v(-1/2+\varepsilon)}L_v^2(1/2+\overline{s_1},\tilde{\pi}),\]
where the implied constant only depends on $\varepsilon$.
Combining bounds for all places $v\mid \q$ and coming back to the global question, we get
\[\langle \overline{\mathrm{Eis}(f_{1/2+s_1})}\mathrm{Eis}(f_{1/2}),\vph(.x^{-1}) \rangle_{\mathrm{reg}} \ll_{\varepsilon,F} \Lambda^2(1/2+\overline{s_1},\tilde{\pi}) N(\q)^{-1/2+\varepsilon}.\]

If $\pi$ is a cuspidal representation, we use the bound on Satake parameters $\max \{\lvert \alpha_1^{r_v}\rvert,\lvert \alpha_2^{r_v}\rvert\}\ll N(\p_v)^{r_v\cdot \theta},$ which implies that
\[\langle \overline{\mathrm{Eis}(f_{1/2+s_1})}\mathrm{Eis}(f_{1/2}),\vph(.x^{-1}) \rangle_{\mathrm{reg}} \ll_{\varepsilon,F} \Lambda^2(1/2+\overline{s_1},\tilde{\pi}) N(\q)^{-1/2+\theta+\varepsilon}.\]
\end{proof}

We have all the tools for the main proposition of this section.
\begin{proof}[Proof of Proposition \ref{Prop: Regularised term}]
    We take $s_1$ close to $0$ and combine the results from the Lemma \ref{Lem: Regularised triple product calc} to get
    \begin{align*}        &\left\lvert\int_{\widehat{\mathbb{X}}_{\mathrm{gen}}^{\mathrm{unr}}}\frac{1}{\lVert \vph_{\pi}\rVert^2_{\pi}}\langle \overline{\mathrm{Eis}(f_{1/2+s_1})}(.x)\mathrm{Eis}(f_{1/2})(.x),\vph_{\pi} \rangle_{\mathrm{reg}} \langle \vph_{\pi},\lvert \vph_0\rvert^2 \rangle\, \mathrm{d}\mu_{\mathrm{aut}}(\pi)\right\rvert \\
    \leq &\int_{\widehat{\mathbb{X}}_{\mathrm{gen}}^{\mathrm{unr}}}\frac{1}{\lVert \vph_{\pi}\rVert^2_{\pi}}\left\lvert\langle \overline{\mathrm{Eis}(f_{1/2+s_1})}(.x)\mathrm{Eis}(f_{1/2})(.x),\vph_{\pi} \rangle_{\mathrm{reg}} \langle \vph_{\pi},\lvert \vph_0\rvert^2 \rangle\right\rvert\, \mathrm{d}\mu_{\mathrm{aut}}(\pi)\\
        \ll_{N,\varepsilon,\vph_0}& N(\q)^{-1/2+\theta+\epsilon} \int_{\widehat{\mathbb{X}}_{\mathrm{gen}}^{\mathrm{unr}}}\frac{1}{\lVert \vph_{\pi}\rVert^2_{\pi}} \Big\lvert\Lambda^2(1/2+\overline{s_1}, \tilde{\pi}) \langle \vph_{\pi},\lvert \vph_0\rvert^2 \rangle\Big\rvert\, \mathrm{d}\mu_{\mathrm{aut}}(\pi).\\
    \end{align*}
We want to show that the final integral converges. To do that we recall that $\vph_{\pi}$ is arithmetically normalised and so, as in Section \ref{section: Calculations on the spectral side}, we obtain that $\lVert \vph_{\pi}\rVert_\pi^2\asymp \Lambda(1,\pi,\mathrm{Ad})$ for $\pi$ cuspidal and $\lVert \vph_{\pi}\rVert^2_{\pi} \asymp \xi^2(1+2it)$ if $\pi=\mathcal{I}(1/2+it)$. In any case, by comparing the Gamma factors and using convexity bound, we get
\[\frac{\Lambda^2(1/2+\overline{s_1}, \tilde{\pi})}{\lVert \vph_{\pi}\rVert_{\pi}^2} \ll C(\pi)^{O(1)}.\]
Finally, by the triple product formula \cite{Ichino2008}, we can realise $\langle \vph_{\pi},\lvert \vph_0\rvert^2 \rangle$ as an $L$--function, namely
\[\langle \vph_{\pi},\lvert \vph_0\rvert^2 \rangle\asymp_{\pi_0} \sqrt{\Lambda(1/2,\pi\otimes \pi_0\otimes\tilde{\pi}_0)}\ll_{N,\pi_0} C(\pi)^{-N},\]
where the upper bound holds for any $N>0.$ We combine the estimates and use the Weyl law \cite{Muller_2008} to conclude that for any $N>2$
\[\int_{\widehat{\mathbb{X}}_{\mathrm{gen}}^{\mathrm{unr}}}\frac{1}{\lVert \vph_{\pi}\rVert^2_{\pi}} \Big\lvert\Lambda^2(1/2+\overline{s_1}, \tilde{\pi}) \langle \vph_{\pi},\lvert \vph_0\rvert^2 \rangle\Big\rvert\, \mathrm{d}\mu_{\mathrm{aut}}(\pi)\ll_{N,\pi_0}  \int_{\widehat{\mathbb{X}}_{\mathrm{gen}}^{\mathrm{unr}}} \lvert C(\pi)\rvert^{-N} \, \mathrm{d}\mu_{\mathrm{aut}}(\pi)\ll_{N} 1,\]
which finishes the proof.
\end{proof}

\section{Local spectral weight}
\begin{proposition}\label{Prop: Local Specwght}
Let $v\mid \q$. For any generic local representation $\pi_v$ of $\G(F_v)$ with $\mathfrak{C}(\pi_{v})\nmid \p_v^{r_v}$ we have $J(f_{1/2,v}(.x_{v}), W_{0,v}, \pi_{v})=0$. In the case when $\mathfrak{C}(\pi_{v})\mid \p_v^{r_v}$ we have
\[J(f_{1/2,v}(.x_{v}), W_{0,v}, \pi_{v})\geq 
    \begin{cases}
         \mathrm{vol}(\mathrm{K}_{\p_v^{r_v}}) \frac{\zeta_{v}(2)}{\zeta_{v}^{3}(1)}\vspace{0.4 pc}; \qquad &\text{if $\pi_v$ is unramified, }\\
         \mathrm{vol}(\mathrm{K}_{\p_v^{r_v}})\frac{L_v(1,\pi\otimes \tilde{\pi})}{\zeta_{v}(1)} \left(1-N(\p_v)^{-1} \lvert \alpha_\pi(v,1)\rvert^2\right); \;&\text{otherwise,}
    \end{cases}\]
where, $\alpha_\pi(v,j)$ are the Satake parameter of $\pi_v$ and we let $\alpha_\pi(v,2)=0$ if $\pi_v$ is ramified. Hence for $N(\p_v)\gg 1$ we have
\[ \mathrm{vol\,}^{-1}(\mathrm{K}_{\p_v^{r_v}}) \zeta_v^2(1)J(f_{1/2,v}(.x_{v}), W_{0,v}, \pi_{v})\geq \frac{1-N(\p_v)^{-1+2\theta} }{1-N(\p_v)^{-1}},\]
where $\theta$ is as in Section \ref{section: Notations}.
\end{proposition}
\begin{proof}
Let $W_v$ be the arithmetically normalised newvector of $\pi_v$, as in \ref{section: Whittaker}. Then $W_v/\lVert W_v\rVert$ is the normalised spherical vector with respect to the inner product defined in equation \eqref{eq: Whittaker inner product}. We can bound trivially
\[J(f_{1/2,v}(.x_{v}), W_{0,v}, \pi_{v})\geq \frac{\lvert \Psi_{v}(f_{1/2,v}(.x_{v}),W_{0,v}, \overline{W_{v}}/\lVert W_v\rVert)\rvert^2}{\lvert L_{v}(1/2, \pi_0\otimes \tilde{\pi})\rvert^2}.\]
We do the Iwasawa decomposition of the zeta integral 
\begin{align*}
&\Psi_{v}(f_{1/2,v}(.x_{v}),W_{0,v}, \overline{W_{v}}/\lVert W_v\rVert)=\frac{1}{\lVert W_v\rVert}\int_{\text{N}(F_v)\bs \text{G}(F_v)}  f_{1/2,v}(gx_v)W_{0,v}(g)\overline{W_v(g)}\; \d g\\
&=\frac{1}{\lVert W_v\rVert}\int_{\mathrm{K}(F_v)}\int_{F_v^{\times}}  
\overline{W_{v}}\Bigg(\begin{pmatrix}
    y& \\
    &1
\end{pmatrix} k_v \Bigg) 
W_{0,v}\Bigg( \begin{pmatrix}
    y& \\
    &1
\end{pmatrix} k_v \Bigg)
f_{1/2,v}\Bigg(\begin{pmatrix}
    y& \\
    &1
\end{pmatrix} k x_v \Bigg)\frac{\d^{\times}y}{\lvert y\rvert}  \; \d k_v\\
&=\frac{1}{\lVert W_v\rVert}\int_{\mathrm{K}(F_v)}\int_{F_v^{\times}} \overline{W_v}\Bigg(\begin{pmatrix}
    y& \\
    &1
\end{pmatrix} k_v \Bigg) W_{0,v}\begin{pmatrix}
    y& \\
    &1
\end{pmatrix}f_{1/2,v}( k x_v) \lvert y\rvert^{-1/2} \, \d^{\times}y \; \d k.
\end{align*}
We need to analyse the local induced vector first. If we write $k=\begin{pmatrix}
    a&b\\
    c&d
\end{pmatrix}$ we get
\begin{align*}
    f_{1/2,v}(kx_v)=&N(\p_v)^{r_v/2}\int_{F_v^{\times}} \Phi_v \Bigg( (0\; t) \begin{pmatrix}
    a&b\\
    c&d
\end{pmatrix}\begin{pmatrix}
    \varpi_v^{-r_v}&\\
    &1
\end{pmatrix}\Bigg) \lvert t\rvert\, \d^{\times} t=N(\p_v)^{r_v/2} 
\int_{F_v^{\times}} \Phi_v ((tc\varpi_v^{-r_v},\, td)) \lvert t\rvert\, \d^{\times} t\\
=& N(\p_v)^{r_v/2}\int_{F_v^{\times}} \mathbf{1}_{\oo_v}(tc\varpi_v^{-r_v})\mathbf{1}_{\oo_v^{\times}}(td) \lvert t\rvert\, \d^{\times} t
\end{align*}
where we recall the choice of $\Phi_v$ from Section \ref{section: Set-up}.
For the integrand to be non-zero we must have $c\in \p_v^{r_v}$, $d\in \oo_v^{\times}$ and $t\in \oo_v^{\times}$ in which case we have 
    $f_{1/2,v}(kx_v)=N(\p_v)^{r_v/2}\mathbf{1}_{\mathrm{K}_{\p_v^{r_v}}}(k)$. So the integral simplifies to 
\begin{align*}
\Psi_{v}(f_{1/2,v}(.x_{v}),W_{0,v}, \overline{W_{v}}) &=\frac{N(\p_v)^{r_v/2}}{\lVert W_v\rVert}\int_{\mathrm{K}_{\p_v^{r_v}}} \int_{F_v^{\times}} \overline{W_v}\left(\begin{pmatrix}
    y& \\
    &1
\end{pmatrix} k \right) W_{0,v}\begin{pmatrix}
    y& \\
    &1
\end{pmatrix} \lvert y\rvert^{-1/2} \,\d^{\times}y \; \d k\\
&=\frac{N(\p_v)^{r_v/2}}{\lVert W_v\rVert}
\int_{\mathrm{K}_{\p_v^{r_v}}} \int_{F_v^{\times}} \overline{W_v}\begin{pmatrix}
    y& \\
    &1
\end{pmatrix} W_{0,v}\begin{pmatrix}
    y& \\
    &1
\end{pmatrix} \lvert y\rvert^{-1/2} \,\d^{\times}y \; \d k\\
&=\frac{N(\p_v)^{r_v/2}}{\lVert W_v\rVert}\text{vol}(\mathrm{K}_{\p_v^{r_v}})\int_{F_v^{\times}} \overline{W_v}\begin{pmatrix}
    y& \\
    &1
\end{pmatrix} W_{0,v}\begin{pmatrix}
    y& \\
    &1
\end{pmatrix} \lvert y\rvert^{-1/2} \,\d^{\times}y,
\end{align*}
where recall the definition of the group $\mathrm{K}_{\p_v^{r_v}}$ from equation \eqref{eq: def of cong subgp}. To calculate the final integral we recall the formula of Miyauchi \eqref{eq: Miyauchi} which lets us write
\begin{align*}
    \int_{F_v^{\times}} \overline{W_v}\begin{pmatrix}
    y& \\
    &1
\end{pmatrix} W_{0,v}\begin{pmatrix}
    y& \\
    &1
\end{pmatrix}&\lvert y\rvert^{-1/2} \,\d^{\times}y=\\
&\sum_{n\geq 0} N(\p_v)^{-n/2} \frac{\alpha_{\tilde{\pi}}(v,1)^{n+1}-\alpha_{\tilde{\pi}}(v,2)^{n+1}}{\alpha_{\tilde{\pi}}(v,1)-\alpha_{\tilde{\pi}}(v,2)}
\frac{\alpha_{\pi_0}(v,1)^{n+1}-\alpha_{\pi_0}(v,2)^{n+1}}{\alpha_{\pi_0}(v,1)-\alpha_{\pi_0}(v,2)}  .
\end{align*}
If we combine all of the geometric series, we get
\begin{equation}\label{eq: local rankin selberg calculation}
    \int_{F_v^{\times}} \overline{W_v}\begin{pmatrix}
    y& \\
    &1
\end{pmatrix} W_{0,v}\begin{pmatrix}
    y& \\
    &1
\end{pmatrix} \lvert y\rvert^{-1/2} \,\d^{\times}y=\frac{1-N(\p_v)^{-1}\alpha_{\tilde{\pi}}(v,1)\alpha_{\tilde{\pi}}(v,2)\alpha_{\pi_0}(v,1)\alpha_{\pi_0}(v,2)}{\prod_{j=1}^2\prod_{i=1}^2(1-N(\p_v)^{-1/2}\alpha_{\tilde{\pi}}(v,i)\alpha_{\pi_0}(v,j))}.
\end{equation}
Since $\pi_{0,v}$ is unramified we have that the Satake parameters factorise $\{\alpha_{\pi_0\otimes \tilde{\pi}}(v,i,j)\}=\{ \alpha_{\pi_0}(v,i)\alpha_{\tilde{\pi}}(v,j)\}$, see \cite[Chapter 14.1.]{Goldfeld2006}, and hence 
\[\prod_{j=1}^2\prod_{i=1}^2(1-N(\p_v)^{-1/2}\alpha_{\tilde{\pi}}(v,i)\alpha_{\pi_0}(v,j))^{-1}=\prod_{j=1}^2\prod_{i=1}^2(1-N(\p_v)^{-1/2}\alpha_{\pi_0\otimes \tilde{\pi}}(v,i,j))^{-1}=L_v(1/2,\pi_0\times \tilde{\pi})\]

If $\pi_v$ is also unramified, then $\alpha_{\pi}(v,1)\alpha_{\pi}(v,2)=1$ and $\lVert W_v\rVert^2=1$, by our choice of inner product \ref{section: Whittaker}, so
\[J(f_{1/2,v}, W_{0,v}, \pi_{v})\geq N(\p_v)^{r_v} \text{vol}^2(\mathrm{K}_{\p_v^{r_v}}) \zeta_v^{-2}(1)= \text{vol}(\mathrm{K}_{\p_v^{r_v}}) \frac{\zeta_v(2)}{\zeta_v^{3}(1)}.\]
After normalisation, we write
\[\text{vol\,}^{-1}(\mathrm{K}_{\p_v^{r_v}}) \zeta_v^2(1) J(f_{1/2,v}, W_{0,v}, \pi_{v})\geq  \frac{\zeta_v(2)}{\zeta_v(1)}.\]

If $\pi_v$ is ramified, then at least one of the Satake parameters is $0$, w.l.o.g. let $\alpha_{\pi}(v,2)=0$. Hence, in equation \eqref{eq: local rankin selberg calculation} we get
\[\int_{F_v^{\times}} \overline{W_v}
    \begin{pmatrix}
        y& \\
        &1
    \end{pmatrix} W_0\begin{pmatrix}
        y& \\
        &1
    \end{pmatrix} \lvert y\rvert^{-1/2}\,\d^{\times}y =L_v(1/2,\pi_0\otimes
    \tilde{\pi})\]
and by the choice of inner product \eqref{eq: Whittaker inner product}
\[\lVert W_v\rVert^2=\frac{\zeta_v(2)}{L_v(1,\pi\otimes \tilde{\pi})}\left(1-N(\p_v)^{-1} \lvert \alpha_{\pi}(v,1)\rvert^2\right)^{-1}.\]
So
\[J(f_{1/2,v}, W_{0,v}, \pi_{v})\geq  \text{vol}(\mathrm{K}_{\p_v^{r_v}})\frac{L_v(1,\pi\otimes \tilde{\pi})}{\zeta_v(1)} \left(1-N(\p_v)^{-1} \lvert \alpha_{\pi}(v,1)\rvert^2\right) \]   and after normalisation we obtain
\[\text{vol\,}^{-1}(\mathrm{K}_{\p_v^{r_v}}) \zeta_v^2(1) J(f_{1/2,v}, W_{0,v}, \pi_{v})\geq L_v(1,\pi\otimes \tilde{\pi}) \frac{1-N(\p_v)^{-1} \lvert \alpha_{\pi}(v,1)\rvert^2}{1-N(\p_v)^{-1}}. \]
We have a trivial estimate $L_v(1,\pi\otimes \tilde{\pi})\geq 1$, and we use an upper bound on Satake parameters $\lvert \alpha_\pi(v,1)\rvert^2\leq N(\p_v)^{2\theta},$ so
    \[\text{vol\,}^{-1}(\mathrm{K}_{\p_v^{r_v}}) \zeta_v^2(1) J(f_{1/2,v}, W_{0,v}, \pi_{v})\geq \frac{1-N(\p_v)^{-1+2\theta} }{1-N(\p_v)^{-1}}.\]
Therefore for any $\pi_v$ (ramified or unramified) such that $N(\p_v)\gg 1$ we have a lower bound
\[\text{vol\,}^{-1}(\mathrm{K}_{\p_v^{r_v}}) \zeta_v^2(1) J(f_{1/2,v}, W_{0,v}, \pi_{v})\geq \frac{1-N(\p_v)^{-1+2\theta} }{1-N(\p_v)^{-1}}.\]
\end{proof}
\begin{proposition}\label{Prop: Lower bound on specwght}
    We have the lower bound for the spectral weight
    \[J_\q(\pi)\gg_{\varepsilon} (1-\varepsilon)^{\omega_F(\q)} ,\]
    for any $\varepsilon>0$ where $\omega_F(\q)$ is the prime ideal omega function defined in equation \eqref{eq: ideal omega function}  .
\end{proposition}
\begin{proof}
    By Proposition \ref{Prop: Local Specwght}, we know that the local terms in the spectral weight satisfy
   \[\prod_{v\mid \q}\text{vol\,}^{-1}(\mathrm{K}_{\p_v^{r_v}}) \zeta_v^2(1) J(f_{1/2,v}(.x_v), W_{0,v}, \pi_{v})\gg \prod_{v\mid \q} \frac{1-N(\p_v)^{-1/2+2\theta} }{1-N(\p_v)^{-1}}\]
   with the constant independent of $\q$. Therefore, we have that for any $\varepsilon>0$
    \[ \prod_{v\mid \q} \frac{1-N(\p_v)^{-1/2+2\theta} }{1-N(\p_v)^{-1}}\gg_{\varepsilon} (1-\varepsilon)^{\omega_F(\q)}, \]
    since $\theta\leq 1/4$.
\end{proof}

\bibliographystyle{alpha}
\bibliography{bfile} 

\begin{thebibliography}{KMV02}

\bibitem[BB11]{BloBru}
Valentin Blomer and Farrell Brumley.
\newblock On the {R}amanujan conjecture over number fields.
\newblock {\em Ann. of Math. (2)}, 174(1):581--605, 2011.

\bibitem[BH12]{BlomHar}
Valentin Blomer and Gergely Harcos.
\newblock A hybrid asymptotic formula for the second moment of {R}ankin-{S}elberg {$L$}-functions.
\newblock {\em Proc. Lond. Math. Soc. (3)}, 105(3):473--505, 2012.

\bibitem[BM24]{BrumMil}
Farrell Brumley and Djordje Milićević.
\newblock Counting cusp forms by analytic conductor.
\newblock {\em Annales Scientifiques de l’École Normale Supérieure}, 57(5):1473--1597, 2024.

\bibitem[Cog08]{CogRS}
James~W. Cogdell.
\newblock Notes on {$L$}-functions for {${\rm GL}_n$}.
\newblock In {\em School on {A}utomorphic {F}orms on {${\rm GL}(n)$}}, volume~21 of {\em ICTP Lect. Notes}, pages 75--158. Abdus Salam Int. Cent. Theoret. Phys., Trieste, 2008.

\bibitem[Gol06]{Goldfeld2006}
Dorian Goldfeld.
\newblock {\em Automorphic Forms and L-Functions for the Group {$GL(n,(\mathbb{R})$}}, volume~99 of {\em Cambridge Studies in Advanced Mathematics}.
\newblock Cambridge University Press, Cambridge, 2006.

\bibitem[HK25]{HumphriesKhan2025SecondMoment}
Peter Humphries and Rizwanur Khan.
\newblock The second moment of {Rankin}–{Selberg} {$L$}-functions in conductor-dropping regimes.
\newblock {\em arXiv preprint}, 2025.
\newblock submitted, 23 pages.

\bibitem[Ich08]{Ichino2008}
Atsushi Ichino.
\newblock Trilinear forms and the central values of triple product {$L$}-functions.
\newblock {\em Duke Mathematical Journal}, 145(2):281--307, 2008.

\bibitem[Jan22]{JanaGLn}
Subhajit Jana.
\newblock The second moment of {$\mathrm{GL}(n)\times\mathrm{GL}(n)$} {R}ankin-{S}elberg {$L$}-functions.
\newblock {\em Forum Math. Sigma}, 10:Paper No. e47, 39, 2022.

\bibitem[JL70]{JacLang}
Herv{\'e} Jacquet and Robert~P. Langlands.
\newblock {\em Automorphic Forms on {$\mathrm{GL}$}(2)}, volume 114 of {\em Lecture Notes in Mathematics}.
\newblock Springer‐Verlag, Berlin / New York, 1970.

\bibitem[JN21]{JanaNunesRec}
Subhajit Jana and Ramon Nunes.
\newblock Spectral reciprocity for {${\rm GL}_n$} and simultaneous non-vanishing of central {$L$}-values.
\newblock {\em arXiv:2111.02297 [math.NT]}, 2021.

\bibitem[JN23]{JanaNunesMom}
Subhajit Jana and Ramon Nunes.
\newblock Moments of {$L$}-functions via the relative trace formula.
\newblock {\em arXiv:2309.06461 [math.NT]}, 2023.

\bibitem[KMV02]{KowMichVan}
Emmanuel Kowalski, Philippe Michel, and Jeffrey VanderKam.
\newblock Rankin-{S}elberg {$L$}-functions in the level aspect.
\newblock {\em Duke Math. J.}, 114(1):123--191, 2002.

\bibitem[Mia24]{Miao}
Xinchen Miao.
\newblock Spectral reciprocity and hybrid subconvexity bound for triple product {$L$}--functions.
\newblock {\em arXiv:2501.04022 [math.NT]}, 2024.

\bibitem[Miy14]{Miyauchi}
Michitaka Miyauchi.
\newblock Whittaker functions associated to newforms for {$\mathrm{GL}(n)$} over {$p$}-adic fields.
\newblock {\em J. Math. Soc. Japan}, 66(1):17--24, 2014.

\bibitem[MV10]{MVGL2}
Philippe Michel and Akshay Venkatesh.
\newblock The subconvexity problem for {${\rm GL}_2$}.
\newblock {\em Publ. Math. Inst. Hautes \'Etudes Sci.}, 111:171--271, 2010.

\bibitem[Mü08]{Muller_2008}
Werner Müller.
\newblock {\em Weyl’s law in the theory of automorphic forms}, page 133–163.
\newblock London Mathematical Society Lecture Note Series. Cambridge University Press, 2008.

\bibitem[Ven10]{Venkatesh2010Sparse}
Akshay Venkatesh.
\newblock Sparse equidistribution problems, period bounds and subconvexity.
\newblock {\em Annals of Mathematics}, 172(2):989--1094, 2010.

\bibitem[Wal92]{Wallach}
Nolan~R. Wallach.
\newblock {\em Real reductive groups. {II}}, volume 132-II of {\em Pure and Applied Mathematics}.
\newblock Academic Press, Inc., Boston, MA, 1992.

\bibitem[Wu19]{Han}
Han Wu.
\newblock Deducing {S}elberg trace formula via {R}ankin-{S}elberg method for {$\rm{GL}_2$}.
\newblock {\em Trans. Amer. Math. Soc.}, 372(12):8507--8551, 2019.

\bibitem[Zac19]{ZachariasI}
Raphaël Zacharias.
\newblock Periods and reciprocity i.
\newblock {\em International Mathematics Research Notices}, 2021(3):2191--2209, 06 2019.

\bibitem[Zac20]{ZachariasII}
Raphaël Zacharias.
\newblock Periods and reciprocity ii.
\newblock {\em arXiv:1912.01512v2 [math.NT]}, 2020.

\end{thebibliography}
\end{document}